\documentclass[final,3p,times]{elsarticle}

\usepackage{amssymb, hyperref, amsmath, algorithm, algpseudocode,amsthm}
\usepackage[T1]{fontenc}
\biboptions{sort&compress}
\hypersetup{hidelinks}
\hypersetup{
colorlinks=true,
linkcolor=black
}

\newcommand{\be}{\begin{equation}}
\newcommand{\ee}{\end{equation}}
\newcommand{\bex}{\begin{equation*}}
\newcommand{\eex}{\end{equation*}}
\newcommand{\bea}{\begin{eqnarray}}
\newcommand{\eea}{\end{eqnarray}}
\newcommand{\beas}{\begin{eqnarray*}}
\newcommand{\eeas}{\end{eqnarray*}}

\newtheorem{theorem}{Theorem}[section]
\newtheorem{definition}[theorem]{Definition}
\newtheorem{proposition}[theorem]{Proposition}
\newtheorem{lemma}[theorem]{Lemma}

\newtheorem{remark}[theorem]{Remark}
\newtheorem{assumption}{Assumption}[section]
\newproof{pf}{Proof}

\makeatletter
\newenvironment{breakablealgorithm}
  {
   \begin{center}
     \refstepcounter{algorithm}
     \hrule height.8pt depth0pt \kern2pt
     \renewcommand{\caption}[2][\relax]{
       {\raggedright\textbf{\ALG@name~\thealgorithm} ##2\par}%
       \ifx\relax##1\relax 
         \addcontentsline{loa}{algorithm}{\protect\numberline{\thealgorithm}##2}%
       \else 
         \addcontentsline{loa}{algorithm}{\protect\numberline{\thealgorithm}##1}%
       \fi
       \kern2pt\hrule\kern2pt
     }
  }{
     \kern2pt\hrule\relax
   \end{center}
  }
\makeatother

\journal{journal}

\begin{document}

\begin{frontmatter}



\title{Robust policy iteration for continuous-time stochastic $H_\infty$ control problem with unknown dynamics}


\author[1]{Zhongshi Sun}
\ead{stone.sun@mail.sdu.edu.cn}

\author[1]{Guangyan Jia\corref{cor1}}
\ead{jiagy@sdu.edu.cn}

\cortext[cor1]{Corresponding author}

\affiliation[1]{organization={Zhongtai Securities Institute for Financial Studies, Shandong University},
            city={Jinan},
            postcode={250100}, 
            state={Shandong},
            country={China}}

\begin{abstract}
In this article, we study a continuous-time stochastic $H_\infty$ control problem based on reinforcement learning (RL) techniques that can be viewed as solving a stochastic linear-quadratic two-person zero-sum differential game (LQZSG). First, we propose an RL algorithm that can iteratively solve stochastic game algebraic Riccati equation based on collected state and control data when all dynamic system information is unknown. In addition, the algorithm only needs to collect data once during the iteration process. Then, we discuss the robustness and convergence of the inner and outer loops of the policy iteration algorithm, respectively, and show that when the error of each iteration is within a certain range, the algorithm can converge to a small neighborhood of the saddle point of the stochastic LQZSG problem. Finally, we applied the proposed RL algorithm to two simulation examples to verify the effectiveness of the algorithm.

\end{abstract}



\begin{keyword}
Reinforcement learning \sep policy iteration \sep stochastic $H_\infty$ control \sep linear-quadratic game \sep robustness.



\end{keyword}

\end{frontmatter}


\section{Introduction}
In the past few decades, the $H_\infty$ control theory has received a lot of attention and development. This theory studies the controller design problems under worst-case scenarios \citep{bacsar2008h}. The method of solving the saddle point of the infinite-horizon zero-sum game problem (ZSG problem) can be directly applied to the $H_{\infty}$ control problem, where the control term tries to minimize the cost function, while the disturbance term tries to maximize it \citep{vamvoudakis2012online}. The characterization of saddle points for control problems requires solving the Hamilton-Jacobi-Bellman-Isaacs (HJI) equations for nonlinear systems \citep{van19922} and the game theoretic algebraic Riccati equations (GAREs) for linear systems \citep{wu2013simultaneous, sun2016linear}. When the game dynamics and cost functions are known, the analytical solutions of the HJI equations are, in general, not available, and one needs to resort to numerical methods such as grid-based approaches \citep{huang2014automation} to solve these partial differential equations (PDEs) approximately.  For linear systems, \citet{lanzon2008computing} proposed a recursive method to solve GAREs in 2008. In 2010, \citet{feng2010iterative} developed a new algorithm to solve a stochastic GARE arising in a linear-quadratic (LQ) stochastic differential game with state-dependent noises. Then \citet{dragan2011computation} extended to the case where the stochastic system is affected by state- and control-dependent white noise.

The development of approximate/adaptive dynamic programming (ADP) \citep{lewis2013reinforcement,jiang2017robust} and reinforcement learning (RL) \citep{sutton2018reinforcement} theories have significant meanings for solving problems of $H_\infty$ control or differential games with uncertain systems. 
Compared with model-based approaches, RL and ADP approaches focus on how to learn the saddle point from past data to reinforce rewards without knowing the structure of the dynamic system.
In 2011, \citet{vrabie2011adaptive} presented an ADP algorithm to find the saddle point solution of an LQZSG problem without requiring information on internal system dynamics. \citet{wu2013simultaneous} proposed a synchronous policy iterative scheme containing only one iterative loop for solving a linear continuous-time $H_\infty$ control in 2013. Afterward, \citet{li2014integral} developed an RL method for an LQZSG problem with completely unknown dynamics. Moreover, \citet{chen2023online} presented two model-free algorithms that do not need to initialize the stabilizing control policy to find the Nash equilibrium solution. However, the ADP algorithms mentioned above for solving the LQZSG problem are all for deterministic systems, and limited work has been done for the stochastic LQZSG problem. Recently, RL and ADP methods for stochastic systems have been of interest to many researchers. (see, for instance, \citep{bian2016adaptive, liu2019stackelberg, wang2020reinforcement,li2022stochastic, pang2022reinforcement,jia2022policy}). In 2022, \citet{guo2022entropy} and \citet{firoozi2022exploratory} studied a class of entropy-regularized mean-field games in which the exploration-exploitation tradeoff for RL is considered.

Moreover, the robustness of the algorithm is the core and hardest problem in RL design. It is possible that the estimation errors caused by sampling will cause the policy iteration (PI) process to diverge. In 2021, \citet{pang2021robust} studied the robustness of the PI algorithm for the continuous-time LQ control problem. \citet{pang2022reinforcement} also analyzed the PI algorithm for optimal stationary control problem of stochastic linear systems is robust with respect to generating estimation errors during the learning process. Then, \citet{cui2023lyapunov} utilized Lypuanov's direct method to analyze the convergence of the PI algorithm under disturbances. Moreover, \citet{cui2023robust} discussed the robustness of an iterative scheme for solving a mixed $H_{2}/H_{\infty}$ control problem.

In this article, we focus on a stochastic continuous-time $H_\infty$ control problem where the diffusion term in the linear stochastic system contains state and control variables. The problem can be viewed as an infinite-horizon stochastic LQZSG problem. The contributions of this paper are stated below.
\begin{itemize}
  \item Based on the model-based PI algorithm for solving the stochastic GARE, we propose a novel RL algorithm. In contrast to numerical methods that require information about all system matrices \citep{dragan2011computation}, this algorithm uses the ADP technique to iteratively solve the stochastic LQZSG problem only by collecting state and input data. In addition, compared with \citep{sun2023reinforcement}, this RL algorithm removes the assumption that partial system information need to be known, so it is a completely model-free algorithm, and the iterative process of the algorithm only needs to collect data once.
  \item We give a robustness analysis of the PI algorithm and prove that when the error generated in each iteration of the RL algorithm is bounded and small, the solution of the algorithm iteration will converge to the small neighborhood of the optimal solution of the stochastic LQZSG problem. Compared with \citep{pang2022reinforcement}, two levels of loops are required in the process of solving stochastic GARE, so the robustness and convergence analysis of the inner loop and the outer loop need to be done, respectively.
  \item We conduct numerical experiments on a 2-dimensional linear system and a 4-dimensional linear two-mass spring system based on the proposed algorithm. The simulation results show that the algorithm can converge to the optimal solution well.
\end{itemize}

The rest of the paper is organized as follows. In Section 2, we briefly describe the stochastic $H_\infty$ control problem and model-based PI algorithm. Section 3 proposes a completely model-free PI for the stochastic LQZSG problem. Section 4 presents a robust PI algorithm and discusses its robustness. Finally, Section 5 shows two numerical simulation results that demonstrate the effectiveness of the proposed RL algorithm. 

Notations: Let $\mathbb{R}^{n\times m}$ be the set of all $n \times m$ real matrices. Let $\mathbb{R}^{n}$ be the $n$-dimensional Euclidean space. $\Vert \cdot \Vert_{F}$ is the Frobenius norm. $\Vert \cdot \Vert_{2}$ is the 2-norm for vectors and the induced 2-norm for matrices. Let $\Vert \cdot \Vert_{\infty}$ denote the supremum norm of a matrix-valued signals, i.e $\Vert \Delta \Vert_{\infty} = sup_{s}\Vert \Delta_{s} \Vert_{F}$. For a matrix $X \in \mathbb{R}^{n \times m}$, $vec(X)$ = $\begin{aligned}\begin{bmatrix} x_{1}^{\top}, x_{2}^{\top}, \cdots x_{m}^{\top}\end{bmatrix}^{\top}\end{aligned}$, where $x_{i}^{\top} \in \mathbb{R}^{n}$ is the $i$-th column of $X$. $\otimes$ indicates the Kronecker product. $I_{n}$ denotes the $n$-dimensional identity matrix. $\mathbb{S}^{n}$ is the set of all symmetric matrices in $\mathbb{R}^{n \times n}$. Define $\mathcal{B}_{r}(X) = \left\{Y \in \mathbb{R}^{m \times n} | \Vert Y -X \Vert_{F} < r \right\}$ and $\bar{\mathcal{B}}_{r}(X)$ as the closure of $\mathcal{B}_{r}(X)$. A continuous function $\kappa : [0,a] \rightarrow [0,\infty)$ is said to be a class $\mathcal{K}$ function if it is strictly increasing and $\kappa(0) = 0$, and a continuous function $\beta: [0,a]\times[0,\infty) \rightarrow [0,\infty)$ is said to be a class $\mathcal{KL}$ function if, for each fixed $t$, the mapping $\beta(\cdot,t)$ is a class $\mathcal{K}$ function and, for every fixed $p$, $\beta(p,t)$ decreases to 0 as $t \rightarrow \infty$.

\label{sec:intro}
\section{Problem formulation and some preliminaries}

Let $(\Omega, \mathcal{F}, \mathbb{F}, \mathbb{P})$ be a complete filtered probability space on which a one-dimensional standard Brownian motion $W(\cdot)$ is defined with $\mathbb{F}=\left\{\mathcal{F}_t\right\}_{t \geqslant 0}$ being its natural filtration augmented by all the $\mathbb{P}-$null sets in $\mathcal{F}$. Consider a class of continuous-time time-invariant stochastic linear dynamical control systems described by
\be
\label{sde}
dX(s) = \left[AX(s)+B_{1}u(s)+B_{2}v(s)\right]d s + \left[CX(s)+Du(s)\right]d W(s),
\ee
where $X(\cdot) \in \mathbb{R}^{n}$ is the system state with initial state $X\left(0\right)=x_{0}$, $u(\cdot) \in \mathbb{R}^{m_{1}}$ is the control input, and $v(\cdot) \in \mathbb{R}^{m_{2}}$ is the external disturbance input. The coefficients $A$, $C \in \mathbb{R}^{n\times n}$, $B_{1}$, $D \in \mathbb{R}^{n\times m_{1}}$ and $B_{2} \in \mathbb{R}^{n\times m_{2}}$ are unknown constant matrices. Let $\mathbb{H}$ be a Euclidean space, and we define the following space:
$L_{\mathbb{F}}^{2}(\mathbb{H}):=\left\{\varphi:[0, \infty) \times \Omega \rightarrow \mathbb{H} \mid \varphi(\cdot)\right.$ is $\mathbb{F}$-progressively measurable process, and $\mathbb{E} \int_{0}^{t}|\varphi(s)|^{2} d s<\infty$ for each $t \ge 0\left.\right\}$.

Clearly, the state equation \eqref{sde} admits a unique solution $X$ for any control pair $(u,v)\in L_{\mathbb{F}}^{2}(\mathbb{R}^{m_{1}})\times L_{\mathbb{F}}^{2}(\mathbb{R}^{m_{2}} )$. 

Define the infinite horizon performance index
\be
\label{per_func}
J\left(x_0; u(\cdot), v(\cdot)\right):=\mathbb{E}\left[\int_0^{\infty}\left(X^{\top}(s) Q X(s)+u^{\top}(s) R u(s)-\gamma^2 v^{\top}(s) v(s)\right) \mathrm{d} s\right]
\ee
with $Q = Q^{\top} \ge 0$, $R = R^{\top} > 0$. Now, we give the definition of mean-square stabilizability.

\begin{definition}
  System \eqref{sde} is called mean-square stabilizable for any initial state $x_{0}$, if there exists a matrix pair $\left(L_{u},L_{v}\right)$ such that the following closed-loop system
  \bex
  \left\{\begin{array}{ll}
    d X(s)=\left[AX(s)+B_{1}L_{u}X(s)+B_{2}L_{v}X(s)\right] d s+\left[CX(s)+DL_{u}X(s)\right] d W(s), ~ s \geq 0 \\
    X_{0}=x
  \end{array}\right.
  \eex
  satisfies $\lim _{s \rightarrow \infty} \mathbb{E}\left[X(s)^{\top}X(s)\right]=0$. In this case, the feedback control pair $\left(u(\cdot),v(\cdot)\right) = \left(L_{u}X(\cdot), L_{v}X(\cdot)\right)$ is called stabilizing, and the matrix pair $\left(L_{u}, L_{v}\right)$ is called a stabilizer of system \eqref{sde}.
\end{definition}

\begin{assumption}
  \label{ass_1}
System \eqref{sde} is mean-square stabilizable.
  \end{assumption}

Under Assumption \eqref{ass_1}, we define the sets of admissible control pairs as
\bex
\mathcal{U}_{a d}:=\left\{\left(u(\cdot),v(\cdot)\right) \in L_{\mathbb{F}}^2\left(\mathbb{R}^{m_{1}}\right) \times L_{\mathbb{F}}^2\left(\mathbb{R}^{m_{2}}\right) \mid \left(u(\cdot),v(\cdot)\right) \text { is stabilizing }\right\} \text {. }
\eex

In the two-player zero-sum differential game, the control policy player $u$ seeks to minimize \eqref{per_func}, while the disturbance policy player $v$ desires to maximize \eqref{per_func}. The goal is to find the saddle point $\left(u^{\ast}(\cdot),v^{\ast}(\cdot)\right) \in \mathcal{U}_{a d}$ such that:
\bex
\sup_{v(\cdot)} \inf_{u(\cdot)}J\left(x_0; u(\cdot), v(\cdot)\right) = \inf_{u(\cdot)} \sup_{v(\cdot)} J\left(x_0; u(\cdot), v(\cdot)\right) = J\left(x_0; u^{\ast}(\cdot), v^{\ast}(\cdot)\right),
\eex
or equivalently, $
J\left(x_0; u^{\ast}(\cdot), v(\cdot)\right) \leq J\left(x_0; u^{\ast}(\cdot), v^{\ast}(\cdot)\right) \leq J\left(x_0; u(\cdot), v^{\ast}(\cdot)\right)$
for any $\left(u(\cdot), v(\cdot)\right) \in \mathcal{U}_{a d}$.

By \citep[Theorem 5.7]{sun2016linear}, when $A$, $B_{1}$, $B_{2}$, $C$ and $D$ are accurately known, the solution to this problem can be found by solving the following GARE:
\be
\label{are_1}
\begin{aligned}
   A^{\top} P+P A+C^{\top} P C+\gamma^{-2} P B_2 B_2^{\top} P-\left[P B_1+C^{\top} P D\right] \times\left[R+D^{\top} P D\right]^{-1}\left[B_1^{\top} P+D^{\top} P C\right]+Q = 0.
  \end{aligned}
\ee

The saddle point of the zero-sum game is 
\bex
\begin{aligned}
  &u^{\ast}(\cdot)= L_{u}^{\ast}X(\cdot) = -\left(R + D^{\top}PD\right)^{-1}\left(B_{1}^{\top}P + D^{\top}PC\right)X(\cdot),\\
  &v^{\ast}(\cdot)= L_{v}^{\ast}X(\cdot) = \gamma^{-2}B_{2}^{\top}PX(\cdot).
\end{aligned}
\eex
and the game value function is $V(x_{0}) := J\left(x_0; u^{\ast}(\cdot), v^{\ast}(\cdot)\right) = x_{0}^{\top}Px_{0}$.

\begin{definition}

A $P \in \mathbb{S}^{n}$ is called a stabilizing solution of \eqref{are_1} if $P$ is a solution and the saddle point $\left(u^{\ast}(\cdot),v^{\ast}(\cdot)\right) \in \mathcal{U}_{a d}$.
\end{definition}

The model-based Algorithm \ref{algo1} aimed at solving \eqref{are_1} iteratively is presented below, which is an equivalent reformulation of the original results in \citep{dragan2011computation, sun2023reinforcement}.

\begin{breakablealgorithm}
  \label{algo1}
  \caption{Model-based Policy Iteration (PI)} 
  \renewcommand{\algorithmicrequire}{\textbf{Initialization:}}
\begin{algorithmic}[1]
  \Require Choose a matrix $P_{u}$ such that there exists a stabilizer $\left(L_{u}, 0\right)$ of the system \eqref{sde} satisfying
  \bex
  \left(A +B_{1}L_{u}\right)^{\top}P_{u} +P_{u}\left(A +B_{1}L_{u}\right)+\left(C+DL_{u}\right)^{\top}P_{u}\left(C+DL_{u}\right)+Q+\gamma^{-2}P_{u}B_{2}B_{2}^{\top}P_{u}+L_{u}^{\top}RL_{u} \leq 0. \eex
  \State Set $L_{v}^{(0)} = 0$, and let $k \leftarrow 0$.
  \Repeat
  \State Let $j \leftarrow 0$, and set $L_{u}^{(k+1,0)} = L_{u}$.
  \Repeat
  \State (Policy evaluation) Solve the following matrix equation for $P_u^{(k+1,j+1)}$:
  \begin{flalign}
    &\left(A  + B_{1}L_{u}^{(k+1,j)} + B_{2}L_{v}^{(k)}\right)^{\top} P_u^{(k+1,j+1)}+P_u^{(k+1,j+1)}  \left(A + B_{1}L_{u}^{(k+1,j)} + B_{2}L_{v}^{(k)} \right)  \notag\\ &+\left(C+DL_{u}^{(k+1,j)}\right)^{\top} P_u^{(k+1,j+1)} \left(C+DL_{u}^{(k+1,j)}\right) +L_{u}^{(k+1,j)^{\top}}RL_{u}^{(k+1,j)}+Q - \gamma^{2}L_{v}^{(k)^{\top}}L_{v}^{(k)}=0,
    \label{PI2}
    \end{flalign}
    
    \State (Policy improvement) Get the improved policy by
    \be
    \label{cf_u}
    L_{u}^{(k+1,j+1)} \leftarrow   -\left(R+D^{\top}P^{(k+1,j+1)}D\right)^{-1}\left(B_{1}^{\top}P^{(k+1,j+1)}+D^{\top}P^{(k+1,j+1)}C\right)
    \ee
    \State $j \leftarrow j+1$
    \Until{$\Vert P_u^{(k+1,j)} - P_u^{(k+1,j-1)}\Vert \leq \epsilon_{1}$}
    \State $P_v^{(k+1)} \leftarrow P_u^{(k+1,j)}$
    \State $L_{v}^{(k+1)} \leftarrow  \gamma^{-2}B_{2}^{\top}P_{v}^{(k+1)}$
    \State $k \leftarrow k+1$
    \Until{$\Vert P_v^{(k)} - P_v^{(k-1)}\Vert \leq \epsilon$}
    \State Output $P_v^{(k)}$ as the stabilizing solution $P^{(\ast)}$ of GARE \eqref{are_1}.
\end{algorithmic}
\end{breakablealgorithm}

Then, the following two theorems guarantee the convergence of Algorithm \ref{algo1}, and the proofs can be found in \citep{dragan2011computation, sun2023reinforcement}.

\begin{theorem}
\label{th2.3}
In the inner loop of Algorithm \ref{algo1}, the following properties hold:

\begin{enumerate}[(i)]
  \item The matrix pair $\left(L_{u}^{(k+1,j)}, L_{v}^{(k)}\right)$ is called a stabilizer of system \eqref{sde} for all $j = 0,1,2,\cdots$;
  \item $P_u^{(k+1,0)} \geq P_u^{(k+1,1)} \geq P_u^{(k+1,2)} \geq \dots \geq P_u^{(k+1,\ast)}$;
  \item $\lim _{j \rightarrow \infty} P_u^{(k+1,j)} = P_u^{(k+1,\ast)}$, $\lim _{j \rightarrow \infty} L_{u}^{(k+1,j)} = L_{u}^{(k+1,\ast)}$.
\end{enumerate}
\end{theorem}

\begin{theorem}
In the outer loop of Algorithm \ref{algo1}, the following properties hold:
\begin{enumerate}[(i)]
  \item The matrix pair $\left(L_{u}^{(k+1,j)}, L_{v}^{(k)}\right)$ is called a stabilizer of system \eqref{sde} for all $k = 0,1,2,\cdots$;
  \item $P_v^{(1)} \leq P_v^{(2)} \leq P_v^{(3)} \leq \dots \leq P_v^{(\ast)}$;
  \item $\lim _{k \rightarrow \infty} P_v^{(k+1)} = P_v^{(\ast)}$, $\lim _{k \rightarrow \infty} L_{v}^{(k+1)} = L_{v}^{(\ast)}$.
\end{enumerate} 
\end{theorem}

Note that the above method needs all knowledge of the system matrices, which are difficult to obtain in the real world. In the next section, we first develop an online model-free RL algorithm for the stochastic LQZSG problem without using the information of the coefficient matrices $A$, $B_{1}$, $B_{2}$, $C$, $D$ in system \eqref{sde}. We then provide the convergence analysis and finally present the online implementation using the least squares method.
\section{PI-based ADP algorithm for the unknown dynamics case}

In this section, we will propose a PI-based ADP algorithm to solve the stochastic LQZSG problem under completely unknown system dynamics. Then, we discuss the convergence of the algorithm and demonstrate its implementation through the least squares method.

For $P \in \mathbb{S}^{n}$, we define an operator as follows:
\bex
  svec(P) :=\begin{bmatrix}p_{11}, 2 p_{12}, \ldots, 2 p_{1 n}, p_{22}, 2 p_{23}, \ldots, 2 p_{n-1, n}, p_{n n}\end{bmatrix}^{\top} \in \mathbb{R}^{\frac{1}{2}n(n+1)}.
\eex

By \citep{murray2002adaptive, li2022stochastic}, there exists a matrix $\mathcal{T} \in \mathbb{R}^{n^{2} \times \frac{1}{2}n(n+1)}$ with rank($\mathcal{T}) = \frac{1}{2}n(n+1)$ such that $vec(P) = \mathcal{T}\times svec(P)$ for any $P \in \mathbb{S}^{n}$. Then, we define $\bar{L} := \left(L^{\top} \otimes L^{\top}\right)\mathcal{T} \in \mathbb{R}^{m^{2} \times \frac{1}{2}n(n+1)}$, which satisfies $L^{\top} \otimes L^{\top} vec(P) = \bar{L}\times svec(P)$ for $L \in \mathbb{R}^{n \times m}$. For example, for vector $x \in \mathbb{R}^{n}$, we have 
\bex
\bar{x} :=\begin{bmatrix}x_{1}^{2}, x_{1} x_{2}, \ldots, x_{1} x_{n}, x_{2}^{2}, x_{2} x_{3}, \ldots, x_{n-1} x_{n}, x_{n}^{2}\end{bmatrix}^{\top} \in \mathbb{R}^{\frac{1}{2}n(n+1)}.
\eex

For brevity, we define the following notation:
\bex
\begin{aligned}
  \phi_{i}^{(k+1,j)} &:= \left[\delta_{xx}^{i},\quad 2I_{xx}^{i}\left(I_{n} \otimes L_{u}^{(k+1,j)^{\top}}\right) - 2I_{xu}^{i},\quad  2I_{xx}^{i}\left(I_{n} \otimes L_{v}^{(k)^{\top}}\right) - 2I_{xv}^{i},\quad I_{xx}^{i}\bar{L}_{u}^{(k+1,j)}-\delta_{uu}^{i}\right],\\  
  \theta_{i}^{(k+1,j)} &:= -I_{xx}^{i}vec\left({L_{u}^{(k+1,j)}}^{\top}RL_{u}^{(k+1,j)} + Q -\gamma^{-2}{L_v^{(k)}}^{\top}L_v^{(k)}\right),
\end{aligned}
\eex
where
\bex
\begin{aligned}
\delta_{xx}^{i} &:= \mathbb{E}^{\mathcal{F}_{t_{i}}}\left[\bar{X}^{\top}(t_{i+1})-\bar{X}^{\top}(t_{i})\right], \quad \delta_{uu}^{i} := \mathbb{E}^{\mathcal{F}_{t_{i}}}\left[\int_{t_{i}}^{t_{i+1}}\bar{u}^{\top}(s)ds\right], \quad  I_{xx}^{i} := \mathbb{E}^{\mathcal{F}_{t_{i}}}\left[\int_{t_{i}}^{t_{i+1}}X^{\top}(s)\otimes X^{\top}(s)ds\right] ,\\   I_{xu}^{i} &:= \mathbb{E}^{\mathcal{F}_{t_{i}}}\left[\int_{t_{i}}^{t_{i+1}}X^{\top}(s)\otimes u^{\top}(s)ds\right],\quad I_{xv}^{i} := \mathbb{E}^{\mathcal{F}_{t_{i}}}\left[\int_{t_{i}}^{t_{i+1}}X^{\top}(s)\otimes v^{\top}(s)ds\right].
\end{aligned}
\eex

\begin{proposition}
\label{pro_ls}
Find $P_u^{(k+1,j+1)}$ from the matrix equation \eqref{PI2} and obtain $L_{u}^{(k+1,j+1)}$ using \eqref{cf_u} in Algorithm \ref{algo1} are equivalent to solving the following equation:
\be
\label{ls_1}
\phi_{i}^{(k+1,j)} \begin{bmatrix} svec\left(P_u^{(k+1,j+1)}\right) \\  vec\left(\tilde{B}_{1}^{(k+1,j+1)} \right)\\ vec\left(\tilde{B}_{2}^{(k+1,j+1)} \right)\\svec\left(\tilde{D}^{(k+1,j+1)} \right)\end{bmatrix} = \theta_{i}^{(k+1,j)},
\ee
where
\bex
\begin{aligned}
&\tilde{B}_{1}^{(k+1,j+1)} := B_{1}^{\top}P_u^{(k+1,j+1)}+D^{\top}P_u^{(k+1,j+1)}C, \quad \tilde{B}_{2}^{(k+1,j+1)} := B_{2}^{\top}P_u^{(k+1,j+1)}, \quad \tilde{D}^{(k+1,j+1)} := D^{\top}P_u^{(k+1,j+1)}D.
\end{aligned}
\eex
\end{proposition}
\begin{proof}
Firstly, we rewrite the original system \eqref{sde} as
\be
\begin{aligned}
dX(s) = &\left[\left(A + B_{1}L_{u}^{(k+1,j)} + B_{2}L_v^{(k)}\right)X(s)+B_{1}\left(u(s)-L_{u}^{(k+1,j)}X(s)\right)+B_{2}\left(v(s)-L_v^{(k)}X(s)\right)\right]d s \\ &+ \left[\left(C+DL_{u}^{(k+1,j)}\right)X(s)+D\left(u(s)-L_{u}^{(k+1,j)}X(s)\right)\right]d W(s).
\end{aligned}
\label{sde_2}
\ee

By \eqref{PI2}, \eqref{cf_u} and \eqref{sde_2}, applying It\^{o}'s formula to $X^{\top}(s)P_u^{(k+1,j+1)}X(s)$ yields

\be
\begin{aligned}
&d\left(X^{\top}(s)P_u^{(k+1,j+1)}X(s)\right) \\=& \left\{\left[\left(A + B_{1}L_{u}^{(k+1,j)} + B_{2}L_v^{(k)}\right)X(s)+B_{1}\left(u(s)-L_{u}^{(k+1,j)}X(s)\right)+B_{2}\left(v(s)-L_v^{(k)}X(s)\right)\right]^{\top}P_u^{(k+1,j+1)}X(s)\right.  \\ &+ X^{\top}(s)P_u^{(k+1,j+1)}\left[\left(A + B_{1}L_{u}^{(k+1,j)} + B_{2}L_v^{(k)}\right)X(s)+B_{1}\left(u(s)-L_{u}^{(k+1,j)}X(s)\right)+B_{2}\left(v(s)-L_v^{(k)}X(s)\right)\right] \\ &+\left.\left[\left(C+DL_{u}^{(k+1,j)}\right)X(s)+D\left(u(s)-L_{u}^{(k+1,j)}X(s)\right)\right]^{\top}P_u^{(k+1,j+1)}\left[\left(C+DL_{u}^{(k+1,j)}\right)X(s)+D\left(u(s)-L_{u}^{(k+1,j)}X(s)\right)\right]\right\}ds \\ &+ (\cdots) dW(s)
\\  =& \bigg\{2\left(u(s)-L_{u}^{(k+1,j)}X(s)\right)^{\top}\left(B_{1}^{\top}P_u^{(k+1,j+1)}+D^{\top}P_u^{(k+1,j+1)}C\right)X(s)+2\left(v(s)-L_v^{(k)}X(s)\right)^{\top}B_{2}^{\top}P_u^{(k+1,j+1)}X(s) \\ & - X^{\top}(s){L_{u}^{(k+1,j)}}^{\top}D^{\top}P_u^{(k+1,j+1)}DL_{u}^{(k+1,j)}X(s) + u^{\top}(s)D^{\top}P_u^{(k+1,j+1)}Du(s) \\ & - X^{\top}(s)\left({L_{u}^{(k+1,j)}}^{\top}RL_{u}^{(k+1,j)} + Q -\gamma^{-2}{L_v^{(k)}}^{\top}L_v^{(k)}\right)X(s)\bigg\}ds + (\cdots) dW(s).
\label{ito_1}
\end{aligned}
\ee

Integrating both sides of \eqref{ito_1} from $t_{i}$ to $t_{i+1}$ and taking conditional expectation $\mathbb{E}^{\mathcal{F}_{t_{i}}}$, we can get
\be
\begin{aligned}
&\mathbb{E}^{\mathcal{F}_{t_{i}}}\left[X^{\top}(t+\Delta t)P_u^{(k+1,j+1)}X(t+\Delta t)- X^{\top}(t)P_u^{(k+1,j+1)}X(t)\right]\\=&\mathbb{E}^{\mathcal{F}_{t_{i}}}\bigg[\int_{t_{i}}^{t_{i+1}}\left\{2\left(u(s)-L_{u}^{(k+1,j)}X(s)\right)^{\top}\left(B_{1}^{\top}P_u^{(k+1,j+1)}+D^{\top}P_u^{(k+1,j+1)}C\right)X(s)+2\left(v(s)-L_v^{(k)}X(s)\right)^{\top}B_{2}^{\top}P_u^{(k+1,j+1)}X(s) \right.\\ & - X^{\top}(s){L_{u}^{(k+1,j)}}^{\top}D^{\top}P_u^{(k+1,j+1)}DL_{u}^{(k+1,j)}X(s) + u^{\top}(s)D^{\top}P_u^{(k+1,j+1)}Du(s) \\ & \left.- X^{\top}(s)\left({L_{u}^{(k+1,j)}}^{\top}RL_{u}^{(k+1,j)} + Q -\gamma^{-2}{L_v^{(k)}}^{\top}L_v^{(k)}\right)X(s)\right\}ds\bigg].
\end{aligned}
\label{ito_2}
\ee

By the vectorization and Kronecker product $\otimes$, we can rewrite \eqref{ito_2} as
\be
\begin{aligned}
&-\mathbb{E}^{\mathcal{F}_{t_{i}}}\left[\int_{t_{i}}^{t_{i+1}}X^{\top}(s)\left({L_{u}^{(k+1,j)}}^{\top}RL_{u}^{(k+1,j)} + Q -\gamma^{-2}{L_v^{(k)}}^{\top}L_v^{(k)}\right)X(s)ds\right]\\=&\mathbb{E}^{\mathcal{F}_{t_{i}}}\bigg[\bar{X}^{\top}(t_{i+1})svec\left(P_u^{(k+1,j+1)}\right)-\bar{X}^{\top}(t_{i})svec\left(P_u^{(k+1,j+1)}\right)\\ &-\int_{t_{i}}^{t_{i+1}}\left(2X^{\top}(s)\otimes \left(u(s)-L_{u}^{(k+1,j)}X(s)\right)^{\top}\right)vec\left(B_{1}^{\top}P_u^{(k+1,j+1)}+D^{\top}P_u^{(k+1,j+1)}C\right)ds\\ &-\int_{t_{i}}^{t_{i+1}}\left(2X^{\top}(s)\otimes\left(v(s)-L_v^{(k)}X(s)\right)^{\top}\right)vec\left(B_{2}^{\top}P_u^{(k+1,j+1)}\right)ds \\ &+ \int_{t_{i}}^{t_{i+1}}\left(X^{\top}(s)\otimes X^{\top}(s)\right) \bar{L}_{u}^{(k+1,j)^{\top}}svec\left(D^{\top}P_u^{(k+1,j+1)}D\right)ds \\& - \int_{t_{i}}^{t_{i+1}}\bar{u}^{\top}(s)svec\left(D^{\top}P_u^{(k+1,j+1)}D\right)ds\bigg] \\=& \phi_{i}^{(k+1,j)} \begin{bmatrix} svec\left(P_u^{(k+1,j+1)}\right) \\  vec\left(\tilde{B}_{1}^{(k+1,j+1)}\right)\\ vec\left(\tilde{B}_{2}^{(k+1,j+1)}\right)\\svec\left(\tilde{D}^{(k+1,j+1)}\right)\end{bmatrix} = \theta_{i}^{(k+1,j)}.
\end{aligned}
\ee
\end{proof}

It can be seen that the desired parameters satisfy \eqref{ls_1}. However, \eqref{ls_1} is only a one-dimensional equation, and we cannot guarantee the uniqueness of the solution. We will use the least squares method to solve this problem.

For any positive integer $N$, denote
 \bex
 \begin{aligned}
 & \Phi^{(k+1,j)} :=\left[\phi_1^{(k+1,j)^{\top}}, \ldots, \phi_N^{(k+1,j)^{\top}}\right]^{\top} = \left[\delta_{xx},\quad 2I_{xx}\left(I_{n} \otimes L_{u}^{(k+1,j)^{\top}}\right) - 2I_{xu},\quad  2I_{xx}\left(I_{n} \otimes L_{v}^{(k)^{\top}}\right) - 2I_{xv},\quad I_{xx}\bar{L}_{u}^{(k+1,j)}-\delta_{uu}\right], \\ &
 \Theta^{(k+1,j)}  :=\left[\theta_1^{(k+1,j)}, \ldots, \theta_N^{(k+1,j)}\right]^{\top} = -I_{xx}vec\left({L_{u}^{(k+1,j)}}^{\top}RL_{u}^{(k+1,j)} + Q -\gamma^{-2}{L_v^{(k)}}^{\top}L_v^{(k)}\right), \\ &
 \delta_{xx} := \left[\delta_{xx}^{0^{\top}}, \ldots, \delta_{xx}^{N-1^{\top}}\right]^{\top}, \quad \delta_{uu} := \left[\delta_{uu}^{0^{\top}}, \ldots, \delta_{uu}^{N-1^{\top}}\right]^{\top}, \quad I_{xx} := \left[I_{xx}^{0^{\top}}, \ldots, I_{xx}^{N-1^{\top}}\right]^{\top},\\ & I_{xu} := \left[I_{xu}^{0^{\top}}, \ldots, I_{xu}^{N-1^{\top}}\right]^{\top}, \quad I_{xv} := \left[I_{xv}^{0^{\top}}, \ldots, I_{xv}^{N-1^{\top}}\right]^{\top}.
 \end{aligned}
 \eex

Then, we have the following $N$-dimensional equation:
\be
\label{ls_3}
\begin{aligned}
\Phi^{(k+1,j)} \begin{bmatrix} svec\left(P_u^{(k+1,j+1)}\right) \\  vec\left(\tilde{B}_{1}^{(k+1,j+1)}\right)\\ vec\left(\tilde{B}_{2}^{(k+1,j+1)}\right)\\svec\left(\tilde{D}^{(k+1,j+1)}\right)\end{bmatrix} = \Theta^{(k+1,j)}.
\end{aligned}
\ee

When the matrix $\Phi^{(k+1,j)^{\top}}$ in \eqref{ls_3} is full column rank, the solution can be obtained using the least squares method as follows:
\be
\label{ls_2}
\begin{aligned}
\begin{bmatrix} svec\left(P_u^{(k+1,j+1)}\right) \\  vec\left(\tilde{B}_{1}^{(k+1,j+1)}\right)\\ vec\left(\tilde{B}_{2}^{(k+1,j+1)}\right)\\svec\left(\tilde{D}^{(k+1,j+1)}\right)\end{bmatrix} = \left(\Phi^{(k+1,j)^{\top}}\Phi^{(k+1,j)} \right)^{-1}\Phi^{(k+1,j)^{\top}}\Theta^{(k+1,j)}.
\end{aligned}
\ee

In practice, at each time point $t_{g}$ $\left(t_{i} \geq t_{g} \geq t_{i+1}, g = 1,2,\cdots, \mathcal{G}\right)$ , we take a control pair $\left(u(t_{g}), v(t_{g})\right)$ as an input to the system \eqref{sde} and obtain $X(t_{g+1})$ by the Euler-Maruyama (EM) method. When we have obtained $\mathcal{H}$ sample paths in a time interval $\left[t_{i},t_{i+1}\right]$, we calculate $\mathbb{E}^{\mathcal{F}_{t_{i}}}\left[ \bar{X}(t_{i+1})^{\top} \right]$ as $\mathbb{E}^{\mathcal{F}_{t_{i}}}\left[ \bar{X}(t_{i+1})^{\top} \right] \approx \frac{1}{\mathcal{H}}\sum_{h = 1}^{\mathcal{H}} \bar{X}^{(h)}(t_{i+1})^{\top}$ and calculate $I_{xx}^{i}$ in \eqref{ls_1} as
\bex
I_{xx}^{i} \approx 
  \frac{1}{\mathcal{H}}\sum_{h = 1}^{\mathcal{H}} \left[\sum_{g=1}^{\mathcal{G}} X^{(h)}(t_{g})^{\top}\otimes X^{(h)}(t_{g})^{\top} \delta t\right].
\eex
Similarly, we can approximate $\delta_{xx}, \delta_{uu}, I_{xx}, I_{xu}$ and $I_{xv}$.

Now, we are ready to give the following ADP algorithm for practical online implementation.

\begin{breakablealgorithm}
  \caption{Model-free PI for LQZSG problem} 
  \label{algo2}
  \renewcommand{\algorithmicrequire}{\textbf{Initialization:}}
\begin{algorithmic}[1]
  \Require Choose a matrix $P_{u}$ such that there exists a stabilizer $\left(L_{u}, 0\right)$ of the system \eqref{sde} satisfying
   \bex x^{\top}P_{u}x-\mathbb{E}\left[X(t)^{\top}P_{u}X(t)\right] \geq \mathbb{E}\left[\int_{0}^{t}\left(X(s)^{\top}QX(s)+X(s)^{\top}L_{u}^{\top}RL_{u}X(s)+\gamma^{-2}X(s)^{\top}P_{u}\Lambda P_{u}X(s)\right)ds\right] .\eex 

  \State Collect data through $\left(u,v\right) = \left(L_{u}X + e_{u},e_{v}\right)$, where $e_{u}$ and $e_{v}$ are exploration signals.
  \State Set $L_{v}^{(0)} = 0$, and let $k \leftarrow 0$.
  \Repeat
  \State Let $j \leftarrow 0$, and set $L_{u}^{(k+1,0)} = L_{u}$.
  \Repeat
  \State Solve for $P_u^{(k+1,j+1)}$, $\tilde{B}_{1}^{(k+1,j+1)}$, $\tilde{B}_{2}^{(k+1,j+1)}$ and $\tilde{D}^{(k+1,j+1)}$ using least squares method \eqref{ls_2}.
  \State \bex L_{u}^{(k+1,j+1)} \leftarrow -\left(R+\tilde{D}^{(k+1,j+1)}\right)^{-1}\left(\tilde{B}_{1}^{(k+1,j+1)}\right)\eex
  \State $j \leftarrow j+1$
  \Until{$\Vert P_u^{(k+1,j)} - P_u^{(k+1,j-1)}\Vert \leq \epsilon_{1}$}
  \State $L_{v}^{(k+1)} \leftarrow  \gamma^{-2}\tilde{B}_{2}^{(k+1,j)}$
  \State $k \leftarrow k+1$
  \Until{$\Vert L_{v}^{(k)} - L_{v}^{(k-1)}\Vert \leq \epsilon$}
  \State Output $P_u^{(k,j)}$ as the stabilizing solution $P^{(\ast)}$ of GARE \eqref{are_1}.
\end{algorithmic}
\end{breakablealgorithm}

\begin{remark}
Since the initialization in Algorithm \ref{algo2} does not know information about matrix $B_2$, we assume that there exists a matrix $\Lambda \in \mathbb{S}^{n}$ satisfying $B_{2}B_{2}^{\top} \leq \Lambda$.
\end{remark}

Next, we show that the convergence of Algorithm \ref{algo2} can be guaranteed under some rank condition. The proof of Lemma \ref{lem_rank} is similar to the proof of \citep[Lemma 6]{jiang2012computational}, and thus is omitted.

\begin{lemma}
\label{lem_rank}
If there exists an integer $l_{0} > 0$, such that, for all $l \geq l_{0}$,
\be
\label{rank}
rank \left(\left[I_{xx},~ I_{xu},~ I_{xv},~ \delta_{uu}\right]\right) = \frac{n(n+1)}{2} + m_{1}n + m_{2}n + \frac{m_{1}(m_{1}+1)}{2},
\ee
then $\Phi^{(k+1,j)}$ has full column rank for all $j \in \mathbb{Z}_{+}$.
\end{lemma}

Next, we prove the convergence of Algorithm \ref{algo2}.
\begin{theorem}
Under rank condition \eqref{rank}, $\left\{P_u^{(k,j)}\right\}_{j=0}^{\infty}$ and $\left\{L_u^{(k,j)}\right\}_{j=0}^{\infty}$ converge to $P_u^{(k,\ast)}$ and $L_u^{(k,\ast)}$($k\ge 1$) in the inner loop, and $\left\{P_u^{(k,\ast)}\right\}_{k=0}^{\infty}$ converges to the stabilizing solution $P^{(\ast)}$ of the GARE \eqref{are_1} in the outer loop .

\end{theorem}

\begin{proof}
Under condition \eqref{rank}, \eqref{ls_2} has unique solutions $P_u^{(k+1,j+1)}$, $\tilde{B}_{1}^{(k+1,j+1)}$, $\tilde{B}_{2}^{(k+1,j+1)}$ and $\tilde{D}^{(k+1,j+1)}$. Therefore, according to Proposition \ref{pro_ls}, Algorithm \ref{algo2} is equivalent to Algorithm \ref{algo1}. 

Then, without considering the estimation error, the proof of convergence of Algorithm \ref{algo2} is similar to that of \citep[Theorem 4.7]{sun2023reinforcement}. 
\end{proof}

It is noted that errors are introduced in the computation of expectations and integrals in $\Phi^{(k+1,j)}$ and $\Theta^{(k+1,j)}$. Therefore, the least squares solution of \eqref{ls_2} is actually
\be
\label{ls_4}
\begin{aligned}
\begin{bmatrix} svec\left(\hat{P}_u^{(k+1,j+1)}\right) \\  vec\left(\hat{\tilde{B}}_{1}^{(k+1,j+1)}\right)\\ vec\left(\hat{\tilde{B}}_{2}^{(k+1,j+1)}\right)\\svec\left(\hat{\tilde{D}}^{(k+1,j+1)}\right)\end{bmatrix} = \left(\Phi^{(k+1,j)^{\top}}\Phi^{(k+1,j)} \right)^{-1}\Phi^{(k+1,j)^{\top}}\Theta^{(k+1,j)}.
\end{aligned}
\ee 
where $\hat{P}_u^{(k+1,j+1)}$, $\hat{\tilde{B}}_{1}^{(k+1,j+1)}$, $\hat{\tilde{B}}_{2}^{(k+1,j+1)}$ and $\hat{\tilde{D}}^{(k+1,j+1)}$ are estimates of $P_u^{(k+1,j+1)}$, $\tilde{B}_{1}^{(k+1,j+1)}$, $\tilde{B}_{2}^{(k+1,j+1)}$ and $\tilde{D}^{(k+1,j+1)}$.

In the next section, we will show that the convergence of the inner and outer loops can be given by Theorem \ref{thm_main} and Theorem \ref{main_thm_out} when the estimation error satisfies certain conditions.

\section{Robustness analysis of PI}
In reality, model-free policy iteration (PI) can hardly be implemented precisely and exactly because of the existence of various errors (possibly) induced by expectation estimation, state estimation, and so on. So it is of great interest to study whether the PI algorithm is robust to errors in the iterative process. To consider the estimation error in Algorithm \ref{algo2}, we present Algorithm \ref{algo3}.

Denote 
\bex
\begin{aligned}
   M\left(P\right) &:= \begin{bmatrix}Q + A^{\top}P + PA  + C^{\top}PC & PB_{1} + C^{\top}PD & PB_{2}  \\ B_{1}^{\top}P+D^{\top}PC & R+D^{\top}PD & O \\ B_{2}^{\top}P & O & -\gamma^{2} \end{bmatrix} \\ &=\left[\begin{array}{c|c|c}
    {[M(P)]_{x x}} & {[M(P)]_{u x}^{\top}} & {[M(P)]_{v x}^{\top}} \\
    \hline[M(P)]_{u x} & {[M(P)]_{u u}} & O \\  \hline {[M(P)]_{v x}} & O & {[M(P)]_{v v}}
    \end{array}\right].
\end{aligned}
\eex

Then, \eqref{PI2} is equivalent to
\bex
\mathcal{H}\left(M\left(P_u^{(k+1,j+1)}\right),L_{u}^{(k+1,j)}, L_{v}^{(k)} \right) = 0,
\eex
where 
\bex
\mathcal{H}\left(M\left(P_u^{(k+1,j+1)}\right),L_{u}^{(k+1,j)}, L_{v}^{(k)} \right) = \begin{bmatrix} I_{n} & L_{u}^{(k+1,j)^{\top}} & L_{v}^{(k)^{\top}} \end{bmatrix} M\left(P_u^{(k+1,j+1)} \right) \begin{bmatrix} I_{n} \\ L_{u}^{(k+1,j)} \\ L_{v}^{(k)} \end{bmatrix}.
\eex

\begin{breakablealgorithm}
  \caption{Robust PI}
  \label{algo3}
  \renewcommand{\algorithmicrequire}{\textbf{Initialization:}}
\begin{algorithmic}[1]
  \Require Choose a matrix $P_{u}$ such that there exists a stabilizer $\left(L_{u}, 0\right)$ of the system \eqref{sde} satisfying
  \bex
  \left(A +B_{1}L_{u}\right)^{\top}P_{u} +P_{u}\left(A +B_{1}L_{u}\right)+\left(C+DL_{u}\right)^{\top}P_{u}\left(C+DL_{u}\right)+Q+\gamma^{-2}P_{u}B_{2}B_{2}^{\top}P_{u}+L_{u}^{\top}RL_{u} \leq 0. \eex
  \State Set $\hat{L}_{v}^{(0)} = 0$, and let $k \leftarrow 0$.
  \Repeat
  \State Let $j \leftarrow 0$, and set $\hat{L}_{u}^{(k+1,0)} = L_{u}$.
  \Repeat
  \State (Policy evaluation) Obtain $\hat{M}^{(k+1,j+1)} = \Delta M^{(k+1,j+1)} + M\left(\hat{P}_u^{(k+1,j+1)}\right) $, where $\Delta M^{(k+1,j+1)} \in \mathbb{S}^{n+m_{1}+m_{2}}$ is a disturbance, $\hat{P}_u^{(k+1,j+1)} \in  \mathbb{S}^{n}$ satisfies
  \be
  \mathcal{H}\left(M\left(\hat{P}_u^{(k+1,j+1)}\right),\hat{L}_{u}^{(k+1,j)}, \hat{L}_{v}^{(k)} \right) =0 
  \ee 
   
    \State (Policy improvement) Get the improved policy by
    \bex
    \hat{L}_{u}^{(k+1,j+1)} \leftarrow   -\left[\hat{M}^{(k+1,j+1)}\right]_{u u}^{-1}\left[\hat{M}^{(k+1,j+1)}\right]_{u x}
    \eex
    \State $j \leftarrow j+1$
    \Until{$\Vert \hat{P}_u^{(k+1,j)} - \hat{P}_u^{(k+1,j-1)}\Vert \leq \epsilon_{1}$}
    \State $\hat{L}_{v}^{(k+1)} \leftarrow  -\left[\hat{M}^{(k+1,j)}\right]_{v v}^{-1}\left[\hat{M}^{(k+1,j)}\right]_{v x}$
    \State $k \leftarrow k+1$
    \Until{$\Vert \hat{L}_{v}^{(k)} -\hat{L}_{v}^{(k-1)}\Vert \leq \epsilon$}
\end{algorithmic}
\end{breakablealgorithm}

\begin{remark}
The error $\Delta M^{(k+1,j+1)} = \hat{M}^{(k+1,j+1)}- M\left(\hat{P}_u^{(k+1,j+1)}\right)$ can arise from a variety of factors. In Algorithm \ref{algo2}, the error factors mainly come from using the average estimated value of the data sample instead of the expected value and using the EM scheme in collecting the data. The formation of error $\Delta M^{(k+1,j+1)}$ in Algorithm \ref{algo3} also includes but is not limited to the following situations: disturbance of input data in dynamic systems and estimation errors of matrices $Q$, $R$,  and $\gamma$ in the cost function in inverse reinforcement learning.
\end{remark}

Next, we will discuss the robustness of the inner and outer loops in Algorithm \ref{algo3} respectively.

\subsection{Robustness analysis of inner loop}

For $X,Y,Z \in \mathbb{R}^{n\times n}$, define

\bex
\begin{aligned}
\mathcal{L}_{X,Z}(Y) &:= X^{\top}Y + YX + Z^{\top}YZ,~~ \mathcal{P}(X):=I_n \otimes X^{\top}+X^{\top} \otimes I_n,~~ \mathcal{Q}(Z):=Z^{\top} \otimes Z^{\top}, \\ 
\mathcal{A}\left(Y\right) &:= A^{(k)} - B\left(R + D^{\top}YD\right)^{-1}\left(B^{\top}Y+D^{\top}YC\right),~~\mathcal{C}\left(Y\right) := C - D\left(R + D^{\top}YD\right)^{-1}\left(B^{\top}Y+D^{\top}YC\right),
\end{aligned}
\eex
where $A^{(k)} = A  + B_{2}L_{v}^{(k)}$.

Then it is easy to check
\be
\label{vec_XZ}
vec\left(\mathcal{L}_{X,Z}(Y)\right) = \left(\mathcal{P}(X)+\mathcal{Q}(Z)\right)vec(Y).
\ee

For example, $Y \in \mathbb{S}^{n}$ is a stabilizing solution of \eqref{are_1}
, then by \citep[Theorem 1]{rami2000linear} we know that $\mathcal{P}\left(\mathcal{A}\left(Y\right)\right)+\mathcal{Q}\left(\mathcal{C}\left(Y\right)\right)$ is Hurwitz, and \eqref{vec_XZ} implies the existence of the inverse operator of $\mathcal{L}_{\mathcal{A}\left(Y\right),\mathcal{C}\left(Y\right)}(\cdot)$ on $\mathbb{S}^{n}$.

In the inner loop of Algorithm \ref{algo1}, we make $P_u^{(k+1,0)} = P_u$ and $L_u^{(k+1,0)} = -\left[M\left(P_u^{(k+1,0)}\right)\right]_{u u}^{-1}\left[M\left(P_u^{(k+1,0)}\right)\right]_{u x}$, and then $\left(L_u^{(k+1,0)}, 0\right)$ remains a stabilizer of the system \eqref{sde}. Inserting \eqref{cf_u} into \eqref{PI2}, the sequence $\left\{P_u^{(k+1,j)}\right\}_{j=0}^{\infty}$ generated by the inner loop satisfies

\be
\label{P_L-}
\begin{aligned}
 P_u^{(k+1,j+1)} =   \mathcal{L}^{-1}_{\mathcal{A}\left(P_{u}^{(k+1,j)}\right),\mathcal{C}\left(P_{u}^{(k+1,j)}\right)}\bigg(-Q^{(k)} - \left[M\left(P_u^{(k+1,j)}\right)\right]_{u x}^{\top}\left[M\left(P_u^{(k+1,j)}\right)\right]_{u u}^{-1}R\left[M\left(P_u^{(k+1,j)}\right)\right]_{u u}^{-1}\left[M\left(P_u^{(k+1,j)}\right)\right]_{u x} \bigg),
\end{aligned}
\ee
where $Q^{(k)} = Q - \gamma^{2}L_{v}^{(k)^{\top}}L_{v}^{(k)}$.

By \eqref{vec_XZ}, the above equation is equivalent to
\bex
\begin{aligned}
 & vec\left(P_u^{(k+1,j+1)}\right) \\  = & \left(\mathcal{P}\left(\mathcal{A}\left(P_{u}^{(k+1,j)}\right)\right)+\mathcal{Q}\left(\mathcal{C}\left(P_{u}^{(k+1,j)}\right)\right)\right)^{-1} vec\bigg(-Q^{(k)} - \left[M\left(P_u^{(k+1,j)}\right)\right]_{u x}^{\top}\left[M\left(P_u^{(k+1,j)}\right)\right]_{u u}^{-1}R\left[M\left(P_u^{(k+1,j)}\right)\right]_{u u}^{-1}\left[M\left(P_u^{(k+1,j)}\right)\right]_{u x} \bigg),
\end{aligned}
\eex

Thus, by \eqref{P_L-}, it is clear that the inner loop of Algorithm \ref{algo1} is a discrete-time nonlinear system, and $P_u^{(k+1,\ast)}$ is an equilibrium by Theorem \eqref{th2.3}. The following lemma states that $P_u^{(k+1,\ast)}$ is actually a locally exponentially stable equilibrium.

\begin{lemma}
  \label{lemma1}
  For any $\sigma \in (0,1)$, there exist $\delta_{0}(\sigma) > 0$ and $c_{0}\left(\delta_{0}\right)>0$ such that for any $P_u^{(k+1,j)} \in \mathcal{B}_{\delta_{0}}\left(P_u^{(k+1,\ast)}\right)$, $\left(L_u^{(k+1,j)},L_v^{(k)}\right)$ is a stabilizer of system \eqref{sde}, and 
  \be
  \left\|P_u^{(k+1,j+1)}-P_u^{(k+1,\ast)}\right\|_F \leq c_0\left\|P_u^{(k+1,j)}-P_u^{(k+1,\ast)}\right\|_F^2.
  \ee

  Especially, \eqref{P_L-} is locally exponentially stable at $P_u^{(k+1,\ast)}$, i.e.,
\be
  \left\|P_u^{(k+1,j+1)}-P_u^{(k+1,\ast)}\right\|_F \leq \sigma\left\|P_u^{(k+1,j)}-P_u^{(k+1,\ast)}\right\|_F.
\ee
\end{lemma}

\begin{proof}
  Since $\left(L_u^{(k+1,\ast)},L_v^{(k)}\right)$ is a stabilizer of system \eqref{sde}, by continuity, there always exists a $\bar{\delta}_{0} >0$ such that $\left(L_u^{(k+1,j)},L_v^{(k)}\right)$ is a stabilizer of system \eqref{sde} for $P_u^{(k+1,j)} \in \bar{\mathcal{B}}_{\bar{\delta}_{0}}\left(P_u^{(k+1,\ast)}\right)$. Suppose $P_u^{(k+1,j)} \in \bar{\mathcal{B}}_{\bar{\delta}_{0}}\left(P_u^{(k+1,\ast)}\right)$. 
  By Theorem \ref{th2.3}, when $\left\{P_u^{(k+1,j)}\right\}_{j=0}^{\infty}$  converges, $P_{u}^{(k+1,\ast)}$ is a solution of GARE as follows
  \be
  \label{are_ast}
  A^{(k)^{\top}}P_{u}^{(k+1,\ast)} + P_{u}^{(k+1,\ast)}A^{(k)} + C^{\top}P_{u}^{(k+1,\ast)}C + Q^{(k)} - \left[M\left(P_u^{(k+1,\ast)}\right)\right]_{u x}^{\top}\left[M\left(P_u^{(k+1,\ast)}\right)\right]_{u u}^{-1}\left[M\left(P_u^{(k+1,\ast)}\right)\right]_{u x} =0.
  \ee
  Adding 
  \bex
  \begin{aligned}
  &-\left[M\left(P_u^{(k+1,j)}\right)\right]_{u x}^{\top}\left[M\left(P_u^{(k+1,j)}\right)\right]_{u u}^{-1}\left[M\left(P_u^{(k+1,\ast)}\right)\right]_{u x}-\left[M\left(P_u^{(k+1,\ast)}\right)\right]_{u x}^{\top}\left[M\left(P_u^{(k+1,j)}\right)\right]_{u u}^{-1}\left[M\left(P_u^{(k+1,j)}\right)\right]_{u x} \\ & + \left[M\left(P_u^{(k+1,j)}\right)\right]_{u x}^{\top}\left[M\left(P_u^{(k+1,j)}\right)\right]_{u u}^{-1}D^{\top}P_u^{(k+1,\ast)}D\left[M\left(P_u^{(k+1,j)}\right)\right]_{u u}^{-1}\left[M\left(P_u^{(k+1,j)}\right)\right]_{u x}
  \end{aligned}
  \eex
  to both sides of \eqref{are_ast} yields
  \be
  \label{L_P}
  \begin{aligned}
  &\mathcal{L}_{\mathcal{A}\left(P_{u}^{(k+1,j)}\right),\mathcal{C}\left(P_{u}^{(k+1,j)}\right)}\left(P_{u}^{(k+1,\ast)}\right) \\ = &-Q^{(k)} - \left[M\left(P_u^{(k+1,j)}\right)\right]_{u x}^{\top}\left[M\left(P_u^{(k+1,j)}\right)\right]_{u u}^{-1}R\left[M\left(P_u^{(k+1,j)}\right)\right]_{u u}^{-1}\left[M\left(P_u^{(k+1,j)}\right)\right]_{u x} + \mathcal{Y}\left(P_u^{(k+1,j)}-P_u^{(k+1,\ast)}\right),
  \end{aligned} 
  \ee

  where 
  \bex
  \begin{aligned}
  &\mathcal{Y}\left(P_u^{(k+1,j)}-P_u^{(k+1,\ast)}\right)\\ =  & \left[M\left(P_u^{(k+1,j)}-P_u^{(k+1,\ast)}\right)\right]_{u x}^{\top}\left[M\left(P_u^{(k+1,j)}\right)\right]_{u u}^{-1}\left[M\left(P_u^{(k+1,j)}-P_u^{(k+1,\ast)}\right)\right]_{u x} \\ & + \left[M\left(P_u^{(k+1,\ast)}\right)\right]_{u x}^{\top}\left[M\left(P_u^{(k+1,\ast)}\right)\right]_{u u}^{-1}D^{\top}\left(P_u^{(k+1,j)}-P_u^{(k+1,\ast)}\right)D\left[M\left(P_u^{(k+1,j)}\right)\right]_{u u}^{-1}\left[M\left(P_u^{(k+1,\ast)}\right)\right]_{u x} \\ & - \left[M\left(P_u^{(k+1,j)}\right)\right]_{u x}^{\top}\left[M\left(P_u^{(k+1,\ast)}\right)\right]_{u u}^{-1}D^{\top}\left(P_u^{(k+1,j)}-P_u^{(k+1,\ast)}\right)D\left[M\left(P_u^{(k+1,j)}\right)\right]_{u u}^{-1}\left[M\left(P_u^{(k+1,j)}\right)\right]_{u x}\\ & + \left[M\left(P_u^{(k+1,j)}\right)\right]_{u x}^{\top}\left[M\left(P_u^{(k+1,\ast)}\right)\right]_{u u}^{-1}D^{\top}\left(P_u^{(k+1,j)}-P_u^{(k+1,\ast)}\right)D\left[M\left(P_u^{(k+1,j)}\right)\right]_{u u}^{-1} \\ & \times D^{\top}\left(P_u^{(k+1,j)}-P_u^{(k+1,\ast)}\right)D\left[M\left(P_u^{(k+1,j)}\right)\right]_{u u}^{-1}\left[M\left(P_u^{(k+1,j)}\right)\right]_{u x}.
  \end{aligned} 
  \eex

  Subtracting \eqref{L_P} from \eqref{P_L-}, we have
  \bex
  P_u^{(k+1,j+1)}-P_u^{(k+1,\ast)} = \mathcal{L}^{-1}_{\mathcal{A}\left(P_{u}^{(k+1,j)}\right),\mathcal{C}\left(P_{u}^{(k+1,j)}\right)}\left(\mathcal{Y}\left(P_u^{(k+1,j)}-P_u^{(k+1,\ast)}\right)\right).
  \eex

  Taking norm on both sides of above euqation and using \eqref{vec_XZ}, we can find a $c_{0}$ such that
  \bex
  \left\|P_u^{(k+1,j+1)}-P_u^{(k+1,\ast)}\right\|_F \leq c_{0}\left\|P_u^{(k+1,j)}-P_u^{(k+1,\ast)}\right\|_F^2.
  \eex
 
\end{proof}

Next, consider the disturbance term $\Delta M^{(k+1,i)}$. We represent Algorithm \ref{algo3} also as a discrete-time nonlinear system and can indicate that the system is locally input-to-state stable by Lemma \ref{lemma1} (see \citep[Definition 2.1.]{jiang2004nonlinear}).

In the inner loop of Algorithm \ref{algo3}, we make $\hat{P}_u^{(k+1,0)} = P_u$, $\hat{L}_u^{(k+1,0)} = -\left[M\left(\hat{P}_u^{(k+1,0)}\right)\right]_{u u}^{-1}\left[M\left(\hat{P}_u^{(k+1,0)}\right)\right]_{u x}$ and $\Delta M^{(k+1,0)} =0$, and then $\left(\hat{L}_u^{(k+1,0)}, 0\right)$ remains a stabilizer of the system \eqref{sde}. The sequence $\left\{\hat{P}_u^{(k+1,j)}\right\}_{j=0}^{\infty}$ generated by the inner loop of Algorithm \ref{algo3} satisfies

\be
\label{P_e}
\begin{aligned}
\hat{P}_u^{(k+1,j+1)} = \mathcal{L}^{-1}_{\mathcal{A}\left(\hat{P}_{u}^{(k+1,j)}\right),\mathcal{C}\left(\hat{P}_{u}^{(k+1,j)}\right)}\bigg( & -Q^{(k)} - \left[M\left(\hat{P}_u^{(k+1,j)}\right)\right]_{u x}^{\top}\left[M\left(\hat{P}_u^{(k+1,j)}\right)\right]_{u u}^{-1}R\left[M\left(\hat{P}_u^{(k+1,j)}\right)\right]_{u u}^{-1}\left[M\left(\hat{P}_u^{(k+1,j)}\right)\right]_{u x} \bigg) \\ &+ \mathcal{E}\left(M\left(\hat{P}_u^{(k+1,j)}\right), \Delta M^{(k+1,j)}\right),
\end{aligned}
\ee
where 
\bex
\begin{aligned}
  & \mathcal{E}\left(M\left(\hat{P}_u^{(k+1,j)}\right), \Delta M^{(k+1,j)}\right) \\ = & \mathcal{L}^{-1}_{\left(A^{(k)} + B_{1} \hat{L}_{u}^{(k+1,j)}\right),\left(C + D \hat{L}_{u}^{(k+1,j)}\right)}\bigg(-Q^{(k)} - \hat{L}_{u}^{(k+1,j)^{\top}}R\hat{L}_{u}^{(k+1,j)} \bigg) \\ &- \mathcal{L}^{-1}_{\mathcal{A}\left(\hat{P}_{u}^{(k+1,j)}\right),\mathcal{C}\left(\hat{P}_{u}^{(k+1,j)}\right)}\bigg(-Q^{(k)} - \left[M\left(\hat{P}_u^{(k+1,j)}\right)\right]_{u x}^{\top}\left[M\left(\hat{P}_u^{(k+1,j)}\right)\right]_{u u}^{-1}R\left[M\left(\hat{P}_u^{(k+1,j)}\right)\right]_{u u}^{-1}\left[M\left(\hat{P}_u^{(k+1,j)}\right)\right]_{u x} \bigg).
\end{aligned}
\eex

The following lemma can be proved in a similar way as in \citep[Lemma 4]{pang2021robust}.
\begin{lemma}
  \label{lemma2}
  There exists a $d(\delta_{0})>0$ such that for all $\hat{P}_u^{(k+1,j)}\in \mathcal{B}_{\delta_{0}}\left(P_u^{(k+1,\ast)}\right)$, $\left(\hat{L}_u^{(k+1,j)},L_v^{(k)}\right)$ is a stabilizer of system \eqref{sde} and $\left[\hat{M}^{(k+1,j)}\right]_{u u}$ is invertible, if $\left\|\Delta M^{(k+1,j)} \right\|_{F} \leq d$.
\end{lemma}

By Lemma \ref{lemma2}, if $\left\| \Delta M\right\|_{\infty} \leq d$, then the sequence $\left\{\hat{P}_u^{(k+1,j)}\right\}_{j=0}^{\infty}$ satisfies \eqref{P_e}. The following lemma gives an upper bound of $\left\|\mathcal{E}\left(M\left(\hat{P}_u^{(k+1,j)}\right), \Delta M^{(k+1,j)}\right) \right\|_{F}$ in terms of $\left\|\Delta M^{(k+1,j)} \right\|_{F}$.

\begin{lemma}
\label{epsilon_ineq}
For any $\hat{P}_u^{(k+1,j)}\in \mathcal{B}_{\delta_{0}}\left(P_u^{(k+1,\ast)}\right)$ and any $c_{2}>0$, there exists a $0<\delta_{1}^{1}(\delta_{0},c_{2})\leq d$, independent of $\hat{P}_u^{(k+1,j)}$, where $d$ is defined in Lemma \ref{lemma2} such that
\bex
\left\|\mathcal{E}\left(M\left(\hat{P}_u^{(k+1,j)}\right), \Delta M^{(k+1,j)}\right) \right\|_{F} \leq c_{3}\left\|\Delta M^{(k+1,j)} \right\|_{F} < c_{2}
\eex
if $\left\|\Delta M^{(k+1,j)} \right\|_{F} < \delta_{1}^{1}$, where $c_{3}(\delta_{0})>0$.
\end{lemma}

\begin{proof}
  For $X \in \mathbb{R}^{n \times n}, Y \in \mathbb{R}^{n \times m}, \Delta X \in \mathbb{R}^{n \times n}$, and $\Delta Y \in \mathbb{R}^{n \times m}$, supposing $X$ and $X+\Delta X$ are invertible, then
  \be
  \label{ineq_1}
  \begin{aligned}
   \left\|X^{-1} Y-(X+\Delta X)^{-1}(Y+\Delta Y)\right\|_F &=  \left\|-X^{-1} \Delta Y+X^{-1} \Delta X(X+\Delta X)^{-1}(Y+\Delta Y)\right\|_F \\
  & \leq\left\|X^{-1}\right\|_F\left(\|\Delta Y\|_F+\left\|(X+\Delta X)^{-1}\right\|_F\|Y+\Delta Y\|_F\|\Delta X\|_F\right) .
  \end{aligned}
  \ee

For any $\hat{P}_u^{(k+1,j)}\in \mathcal{B}_{\delta_{0}}\left(P_u^{(k+1,\ast)}\right)$ and $\left\|\Delta M^{(k+1,j)} \right\|_{F} \leq d$ , by continuity and Lemma \ref{lemma2}, we have from \eqref{ineq_1} 
\be
\label{ineq_c4}
\begin{aligned}
& \left\| \left[M\left(\hat{P}_u^{(k+1,j)}\right)\right]_{u u}^{-1}\left[M\left(\hat{P}_u^{(k+1,j)}\right)\right]_{u x} - \left[\hat{M}^{(k+1,j)}\right]_{u u}^{-1}\left[\hat{M}^{(k+1,j)}\right]_{u x}\right\|_{F} \\  \leq & \left\| \left[M\left(\hat{P}_u^{(k+1,j)}\right)\right]_{u u}^{-1}\right\|_{F} \bigg( \left\|\Delta M^{(k+1,j)} \right\|_{F} + \left\| \left[\hat{M}^{(k+1,j)}\right]_{u u}^{-1}\right\|_{F} \left\| \left[\hat{M}^{(k+1,j)}\right]_{u x}^{-1}\right\|_{F} \left\|\Delta M^{(k+1,j)} \right\|_{F} \bigg)\\ \leq & c_{4}\left\|\Delta M^{(k+1,j)} \right\|_{F}
\end{aligned}
\ee

for some $c_{4}(\delta_{0},d) > 0 $. Using \eqref{ineq_c4}, it is easy to check
\bex
\left\|\mathcal{P}\left(A^{(k)} - B\left[\hat{M}^{(k+1,j)}\right]_{u u}^{-1}\left[\hat{M}^{(k+1,j)}\right]_{u x} \right) - \mathcal{P}\left(\mathcal{A}\left(\hat{P}_u^{(k+1,j)}\right)\right)\right\|_{F}\leq c_{5}\left\|\Delta M^{(k+1,j)} \right\|_{F}
\eex

\bex
\left\|\mathcal{Q}\left(C - D\left[\hat{M}^{(k+1,j+1)}\right]_{u u}^{-1}\left[\hat{M}^{(k+1,j)}\right]_{u x} \right) - \mathcal{Q}\left(\mathcal{C}\left(\hat{P}_u^{(k+1,j)}\right)\right)\right\|_{F}\leq c_{6}\left\|\Delta M^{(k+1,j)} \right\|_{F}
\eex

\bex
\begin{aligned}
\Bigg\| & vec \bigg(\left[M\left(\hat{P}_u^{(k+1,j)}, L_{v}^{(k)}\right)\right]_{u x}^{\top}\left[M\left(\hat{P}_u^{(k+1,j)}, L_{v}^{(k)}\right)\right]_{u u}^{-1}R\left[M\left(\hat{P}_u^{(k+1,j)}, L_{v}^{(k)}\right)\right]_{u u}^{-1}\left[M\left(\hat{P}_u^{(k+1,j)}, L_{v}^{(k)}\right)\right]_{u x} \\ &- \left[\hat{M}^{(k+1,j)}\right]_{u x}^{\top}\left[\hat{M}^{(k+1,j)}\right]_{u u}^{-1}R\left[\hat{M}^{(k+1,j)}\right]_{u u}^{-1}\left[\hat{M}^{(k+1,j)}\right]_{u x}\bigg)\Bigg\|_{F} \leq c_{7}\left\|\Delta M^{(k+1,j)} \right\|_{F}
\end{aligned}
\eex

for some $c_5(\delta_{0},d)>0, c_6(\delta_{0},d)>0$ and $c_7(\delta_{0},d)>0$. Then by continuity, \eqref{vec_XZ} and \eqref{ineq_1} 
\bex
\begin{aligned}
& \left\| \mathcal{E}\left(M\left(\hat{P}_u^{(k+1,j)}, L_{v}^{(k)}\right), \Delta M^{(k+1,j)}\right) \right\|_{F}\\ \leq & \left\| \left(\mathcal{P}\left(A^{(k)}-B\left[\hat{M}^{(k+1,j)}\right]_{u u}^{-1}\left[\hat{M}^{(k+1,j)}\right]_{u x}\right)+ \mathcal{Q}\left(C-D\left[\hat{M}^{(k+1,j)}\right]_{u u}^{-1}\left[\hat{M}^{(k+1,j)}\right]_{u x}\right)\right)^{-1}  \right\|_{F}  \\ &\Bigg(c_{7} + (c_{5}+c_{6})  \Big\| \left(\mathcal{P}\left(\mathcal{A}\left(P_{u}^{(k+1,j)}\right)\right) + \mathcal{Q}\left(\mathcal{C}\left(P_{u}^{(k+1,j)}\right)\right)\right)^{-1}\Big\|_{F} \left\|-Q^{(k)} - \left[M\left(\hat{P}_u^{(k+1,j)}, L_{v}^{(k)}\right)\right]_{u x}^{\top}\left[M\left(\hat{P}_u^{(k+1,j)}, L_{v}^{(k)}\right)\right]_{u u}^{-1} \right.\\ & \left. R \left[M\left(\hat{P}_u^{(k+1,j)}, L_{v}^{(k)}\right)\right]_{u u}^{-1}\left[M\left(\hat{P}_u^{(k+1,j)}, L_{v}^{(k)}\right)\right]_{u x}\right\|_{F}\Bigg)\left\|\Delta M^{(k+1,j)} \right\|_{F} \leq c_{3}(\delta_{0})\left\|\Delta M^{(k+1,j)} \right\|_{F}
\end{aligned}
\eex

Choosing $0 \leq \delta_{1}^{1} \leq d$ such that $c_{3}\delta_{1}^{1} < c_{2}$ completes the proof.
\end{proof}
\begin{lemma}
  \label{main_lemma1}
  For $\sigma$ and its associated $\delta_{0}$ in Lemma \ref{lemma1}, there exists $\delta_{1}(\delta_{0}) > 0$ such that if $\left\| \Delta M\right\|_{\infty} < \delta_{1}, \hat{P}_u^{(k+1,0)}\in \mathcal{B}_{\delta_{0}}\left(P_u^{(k+1,\ast)}\right)$

  \begin{enumerate}[(i)]
    \item $\left[\hat{M}^{(k+1,j)}\right]_{u u}$ is invertible and $\left(\hat{L}_u^{(k+1,j-1)},\hat{L}_v^{(k)}\right)$ is a stabilizer of system \eqref{sde}, ${\forall} j \in \mathbb{Z}_{+}, j > 0 $;
    \item the following local input-to-state stability (ISS) holds: \bex \left\| \hat{P}_u^{(k+1,j)} - P_u^{(k+1,\ast)} \right\|_{F} \leq  \beta \left(\left\| \hat{P}_u^{(k+1,0)} - P_u^{(k+1,\ast)} \right\|_{F}, j\right) + \kappa \left(\left\| \Delta M\right\|_{\infty} \right),\eex where class $\mathcal{KL}$ function $\beta (y,j) = \sigma^{j}y$, class $\mathcal{K}$ function $\kappa(y)=c_{3}y/(1-\sigma),y \in \mathbb{R}$ and $c_{3}(\delta_{0}) > 0$;
    \item $\lim _{j \rightarrow \infty}\left\| \Delta M_{j}\right\|_{\infty}= 0$  implies $\lim _{j \rightarrow \infty}\left\| \hat{P}_u^{(k+1,j)} - P_u^{(k+1,\ast)} \right\|_{F} = 0$.
  \end{enumerate}
\end{lemma}
\begin{proof}
Let $c_{2} = (1 - \sigma)\delta_{0}$ in Lemma \eqref{epsilon_ineq} and $\delta_{1}$ be equal to the $\delta_{1}^{1}$ associated with $c_{2}$.

For any $j \in \mathbb{Z}_{+}$ if $\hat{P}_u^{(k+1,j)}\in \mathcal{B}_{\delta_{0}}\left(P_u^{(k+1,\ast)}\right)$, then 

\begin{align}
\left\|\hat{P}_u^{(k+1,j+1)}-P_u^{(k+1,\ast)}\right\|_F \leq &\left\|\mathcal{E}\left(M\left(\hat{P}_u^{(k+1,j)}\right), \Delta M^{(k+1,j)}\right)\right\|_F + \bigg\|\mathcal{L}^{-1}_{\mathcal{A}\left(\hat{P}_{u}^{(k+1,j)}\right),\mathcal{C}\left(\hat{P}_{u}^{(k+1,j)}\right)}\bigg(-Q^{(k)} - \left[M\left(\hat{P}_u^{(k+1,j)}\right)\right]_{u x}^{\top} \notag\\ & \left[M\left(\hat{P}_u^{(k+1,j)}\right)\right]_{u u}^{-1}R\left[M\left(\hat{P}_u^{(k+1,j)}\right)\right]_{u u}^{-1}\left[M\left(\hat{P}_u^{(k+1,j)}\right)\right]_{u x} \bigg) - P_u^{(k+1,\ast)} \bigg\|_F \notag \\ \label{last3}  \leq & \sigma \left\|\hat{P}_u^{(k+1,j)}-P_u^{(k+1,\ast)}\right\|_F + c_{3}\left\|\Delta M^{(k+1,j)} \right\|_{F} \\ \label{last2} \leq & \sigma \left\|\hat{P}_u^{(k+1,j)}-P_u^{(k+1,\ast)}\right\|_F + c_{3}\left\|\Delta M \right\|_{\infty} \\ \label{last1}< & \sigma \delta_{0} + c_{3}\delta_{1} < \sigma \delta_{0} + c_{2} = \delta_{0},
\end{align}
where \eqref{last3} and \eqref{last1} are due to Lemmas \ref{lemma1} and \ref{epsilon_ineq}. By induction, (i) in Lemma \ref{main_lemma1} is proved.

The proof of (ii) in Lemma \ref{main_lemma1} comes directly from repeating \eqref{last2} as
\bex
\begin{aligned}
\left\|\hat{P}_u^{(k+1,j)}-P_u^{(k+1,\ast)}\right\|_F &\leq \sigma^2 \left\|\hat{P}_u^{(k+1,j-2)}-P_u^{(k+1,\ast)}\right\|_F + c_{3}(\sigma +1)\left\|\Delta M \right\|_{\infty} \leq \dots \\ & < \sigma^j \left\|\hat{P}_u^{(k+1,0)}-P_u^{(k+1,\ast)}\right\|_F + \frac{c_{3}}{1-\sigma}\left\|\Delta M \right\|_{\infty}.
\end{aligned}
\eex

Then, (iii) can be proved in a similar way to that of (iii) in \citep[Lemma 2]{pang2021robust}.

\end{proof}

In Lemma \ref{main_lemma1}, it is required that $\hat{P}_u^{(k+1,0)}\in \mathcal{B}_{\delta_{0}}\left(P_u^{(k+1,\ast)}\right)$; thus, by definition,  $\hat{L}_u^{(k+1,0)}$ must be in a neighborhood of $\hat{L}_u^{(k+1,\ast)}$. The following theorem removes this restriction, where $\hat{L}_u^{(k+1,0)}$ can be any stabilizing control gain.

\begin{theorem}
\label{thm_main}
For any given stabilizer $\left(\hat{L}_u^{(k+1,0)},\hat{L}_v^{(k)}\right)$ of system \eqref{sde} and any $\epsilon > 0$, there exists $\delta_{2}\left(\epsilon, \hat{L}_u^{(k+1,0)}\right) > 0$ such that if $Q^{(k)}>0$ and $\left\| \Delta M\right\|_{\infty} < \delta_{2}$:
\begin{enumerate}[(i)]
  \item $\left[\hat{M}^{(k+1,j)}\right]_{u u}$ is invertible and $\left(\hat{L}_u^{(k+1,j)},\hat{L}_v^{(k)}\right)$ is a stabilizer of system \eqref{sde},$\left\| \hat{P}_u^{(k+1,j)} \right\|_{F} < M_{0}$,${\forall} j \in \mathbb{Z}_{+}, j >0 $, where $M_{0}(\delta_{2})>0$;
  \item $\lim sup _{j \rightarrow \infty}\left\| \hat{P}_u^{(k+1,j)} - P_u^{(k+1,\ast)} \right\|_{F} < \epsilon$;
  \item $\lim _{j \rightarrow \infty}\left\| \Delta M_{j}\right\|_{\infty}= 0$  implies $\lim _{j \rightarrow \infty}\left\| \hat{P}_u^{(k+1,j)} - P_u^{(k+1,\ast)} \right\|_{F} = 0$.
\end{enumerate} 

\end{theorem}

\begin{proof}

Note that all the conclusions of Theorem \ref{thm_main} can be implied by Lemma \ref{main_lemma1} if $\delta_{2} < min(\kappa^{-1}(\epsilon),\delta_{1})$, and $\hat{P}_u^{(k+1,1)}\in \mathcal{B}_{\delta_{0}}\left(P_u^{(k+1,\ast)}\right)$ for Algorithm \ref{algo3}. Thus, the proof of Theorem \ref{thm_main} reduces to the proof of the following Lemma \ref{th_mainlemma}.
\end{proof}

\begin{lemma}
\label{th_mainlemma}
Given a stabilizer $\left(\hat{L}_u^{(k+1,0)},\hat{L}_v^{(k)}\right)$, there exists a $0<\delta_{2}<min(\kappa^{-1}(\epsilon),\delta_{1})$ and a $\bar{j} \in \mathbb{Z}_{+}$ such that $\hat{P}_u^{(k+1,\bar{j})}\in \mathcal{B}_{\delta_{0}}\left(P_u^{(k+1,\ast)}\right)$, as long as $\left\| \Delta M\right\|_{\infty} < \delta_{2}$.
\end{lemma}

The proof is very close to that of \citep[Theorem 2]{pang2021robust} and \citep[Corollary 1]{pang2022reinforcement}, so it is omitted.

\subsection{Robustness analysis of outer loop}

According to \citep{dragan2011computation, sun2023reinforcement} and Algorithm \ref{algo1}, the iteration process of $P_v^{(k)}$ in the outer loop satisfies

\begin{flalign*}
  &\left(A + \gamma^{-2}B_{2}B_{2}^{\top}P_v^{(k)} \right)^{\top} P_v^{(k+1)}+P_v^{(k+1)}  \left(A + \gamma^{-2}B_{2}B_{2}^{\top}P^{(k)} \right)  +C^{\top} P_v^{(k+1)} C \notag\\ &
  -\left(P_v^{(k+1)} B_{1}+C^{\top} P_v^{(k+1)} D\right)\left(R+D^{\top} P_v^{(k+1)} D\right)^{-1} 
   \left(B_{1}^{\top} P_v^{(k+1)}+D^{\top} P_v^{(k+1)} C\right)  +Q - \gamma^{-2}P^{(k)}B_{2}B_{2}^{\top}P_v^{(k)}=0 .
\end{flalign*}
The above equation can be written as
\begin{flalign}
  &\left(A + B_{1}L_u^{(k+1)}+ B_{2}L_v^{(k)} \right)^{\top} P_v^{(k+1)}+P_v^{(k+1)}  \left(A + B_{1}L_u^{(k+1)}+ B_{2}L_v^{(k)} \right)  +\left(C + DL_u^{(k+1)}\right)^{\top} P_v^{(k+1)} \left(C + DL_u^{(k+1)}\right) \notag\\ &
  +L_u^{(k+1)^{\top}}RL_u^{(k+1)} + Q - \gamma^{2}L_v^{(k)^{\top}}L_v^{(k)} =0 .
  \label{lya_out}
\end{flalign}
where $L_u^{(k+1)} := - \left[M\left(P_v^{(k+1)}\right)\right]_{u u}^{-1}\left[M\left(P_v^{(k+1)}\right)\right]_{u x}$, and $L_v^{(k)} := - \left[M\left(P_v^{(k)}\right)\right]_{v v}^{-1}\left[M\left(P_v^{(k)}\right)\right]_{v x}$.

\begin{lemma}
\label{out_lem1}
For any $\delta_{3} > 0$, define $\mathcal{D}_{\delta_{3}}\left(P_v^{(\ast)}\right) = \Big\{L_{v}^{(k)} \in \mathbb{R}^{m\times n} | (L_{u},L_{v}^{(k)})$ is a stabilizer of system \eqref{sde}, and $Tr(P_v^{(\ast)} - P_v^{(k+1)}) \leq \delta_{3}\Big\}$. For any $L_{v}^{(k)} \in \mathcal{D}_{\delta_{3}}\left(P_v^{(\ast)}\right)$, there exists a $b(\delta_{3})>0$, such that $\left\|P_v^{(\ast)} - P_v^{(k+1)} \right\|_{F} \leq b(\delta_{3})\left\|\Xi^{(k+1)}\right\|_{F}$, where $\Xi^{(k+1)} = \gamma^{2}\left(L_v^{(k+1)}- L_v^{(k)}\right)^{\top}\left(L_v^{(k+1)}- L_v^{(k)}\right)$.
\end{lemma}
\begin{proof}
Define $A_{(\ast)} = A + B_{1}L_u^{(\ast)} + B_{2}L_v^{(\ast)}$ and $C_{(\ast)} = C + DL_u^{(\ast)}$. Then \eqref{lya_out} becomes
\begin{flalign}
  & A_{(\ast)}^{\top} P_v^{(k+1)} + P_v^{(k+1)}A_{(\ast)} + C_{(\ast)}^{\top} P_v^{(k+1)}C_{(\ast)} - \gamma^{2}\left(L_v^{(\ast)} - L_v^{(k)}\right)^{\top}L_v^{(k+1)} - \gamma^{2} L_v^{(k+1)^{\top}}\left(L_v^{(\ast)} - L_v^{(k)}\right)- \gamma^{2}L_v^{(k)^{\top}}L_v^{(k)} \notag\\ & + \left(L_u^{(\ast)} - L_u^{(k+1)}\right)^{\top} \left(R+D^{\top}P_v^{(k+1)}D\right)L_u^{(k+1)} + L_u^{(k+1)^{\top}}\left(R+D^{\top}P_v^{(k+1)}D\right)L_u^{(\ast)} - L_u^{(\ast)^{\top}}DP_v^{(k+1)}DL_u^{(\ast)} + Q  = 0.
  \label{lya_out_re}
\end{flalign}

In addition, \eqref{are_1} can be rewritten as
\be
\label{are_1_re}
A_{(\ast)}^{\top} P_v^{(\ast)} + P_v^{(\ast)}A_{(\ast)} + C_{(\ast)}^{\top}P_v^{(\ast)}C_{(\ast)} + L_u^{(\ast)^{\top}}RL_u^{(\ast)} + Q - \gamma^{2} L_v^{(\ast)^{\top}}L_v^{(\ast)} = 0.
\ee

Subtracting \eqref{are_1_re} from \eqref{lya_out_re}, we have
\begin{flalign*}
&A_{(\ast)}^{\top} \left(P_v^{(\ast)}- P_v^{(k+1)}\right) +  \left(P_v^{(\ast)} - P_v^{(k+1)}\right)A_{(\ast)} + C_{(\ast)}^{\top} \left(P_v^{(\ast)} - P_v^{(k+1)}\right)C_{(\ast)} - \gamma^{2} \left(L_v^{(\ast)}- L_v^{(k+1)}\right)^{\top}\left(L_v^{(\ast)}- L_v^{(k+1)}\right) \notag \\ &+ \gamma^{2}\left(L_v^{(k+1)}- L_v^{(k)}\right)^{\top}\left(L_v^{(k+1)}- L_v^{(k)}\right) + \left(L_u^{(\ast)}- L_u^{(k+1)}\right)^{\top}\left(R+D^{\top}P_v^{(k+1)}D\right)\left(L_u^{(\ast)}- L_u^{(k+1)}\right) = 0.
\end{flalign*}

Let $\tilde{P}_{v}^{(k+1)} = P_v^{(\ast)} - P_v^{(k+1)}$. Then by \eqref{ineq_1}, the above equation can be rewritten as
\be
\label{lya_re2}
  A_{(\ast)}^{\top}\tilde{P}_{v}^{(k+1)} + \tilde{P}_{v}^{(k+1)}\left(A + B_{1}L_u^{(k+1)} + B_{2}L_v^{(k+1)}\right) + C_{(\ast)}^{\top} \tilde{P}_{v}^{(k+1)} \left(C + DL_u^{(k+1)}\right) = -\gamma^{2}\left(L_v^{(k+1)}- L_v^{(k)}\right)^{\top}\left(L_v^{(k+1)}- L_v^{(k)}\right).
\ee

Vectorizing both sides of \eqref{lya_re2} yields $\mathcal{K}\left(P_{v}^{(k+1)}\right) vec \left(\tilde{P}_{v}^{(k+1)}\right) = -vec\left(\Xi^{(k+1)}\right)$, where
\bex
\mathcal{K}\left(P_{v}^{(k+1)}\right) = I_{n} \otimes A_{(\ast)}^{\top} + \left(A + B_{1}L_u^{(k+1)} + B_{2}L_v^{(k+1)}\right)^{\top} \otimes I_{n} + C_{(\ast)}^{\top} \otimes \left(C+DL_u^{(k+1)}\right)^{\top}.
\eex

Since $I_{n} \otimes A_{(\ast)}^{\top} + A_{(\ast)}^{\top} \otimes I_{n} + C_{(\ast)}^{\top} \otimes  C_{(\ast)}^{\top}$ is Hurwitz, there exists a $\delta_{4}>0$ such that $\mathcal{K}\left(P_{v}^{(k+1)}\right)$ is invertible for all $L_{v}^{(k)} \in \mathcal{D}_{\delta_{3}}\left(P_v^{(\ast)}\right) \cap \mathcal{B}_{\delta_{4}}\left(L_{v}^{(\ast)}\right)$. Therefore, for any $L_{v}^{(k)} \in \mathcal{D}_{\delta_{3}}\left(P_v^{(\ast)}\right) \cap \mathcal{B}_{\delta_{4}}\left(L_{v}^{(\ast)}\right)$, it follows that
\bex
\left\|\tilde{P}_{v}^{(k+1)} \right\|_{F} \leq \underline{\sigma}^{-1}\left(\mathcal{K}\left(P_{v}^{(k+1)}\right)\right) \left\|\Xi^{(k+1)} \right\|_{F},
\eex
where $\underline{\sigma}(\cdot)$ is the minimum singular value of a matrix.

On the other hand, for any $L_{v}^{(k)} \in \mathcal{D}_{\delta_{3}}\left(P_v^{(\ast)}\right) \cap \mathcal{B}^{c}_{\delta_{4}}\left(L_{v}^{(\ast)}\right)$, where $\mathcal{B}^{c}_{\delta_{4}}$ is a complement of $\mathcal{B}_{\delta_{4}}$, $\Xi^{(k+1)} \neq 0$ and there exists a constant $b_{1} >0$ such that $\left\|\Xi^{(k+1)} \right\|_{F} \geq b_{1}$. Thus, we have 
\bex
\left\|\tilde{P}_{v}^{(k+1)} \right\|_{F} \leq Tr\left(P_v^{(\ast)}\right)  \leq \frac{\delta_{3}+ Tr\left(P_v^{(k+1)}\right)}{b_{1}}\left\|\Xi^{(k+1)} \right\|_{F}\leq \frac{\delta_{3}+ Tr\left(P_v^{(\ast)}\right)}{b_{1}}\left\|\Xi^{(k+1)} \right\|_{F}.
\eex

Suppose that $b_{2} = max_{L_{v}^{(k)}\in \mathcal{D}_{\delta_{3}}\left(P_v^{(\ast)}\right) \cap \mathcal{B}_{\delta_{4}}\left(L_{v}^{(\ast)}\right)}\underline{\sigma}^{-1}\left(\mathcal{K}\left(P_{v}^{(k+1)}\right)\right)$ and $b\left(\delta_{3}\right) = max\left\{b_{2},\frac{\delta_{3}+ Tr\left(P_v^{(\ast)}\right)}{b_{1}} \right\}$, then the proof is complete.
\end{proof}

\begin{lemma}
  \label{lem_Ha}
  Let $	\Psi = \left\{\Psi(s); s \geq 0\right\}$ be the solution to the matrix SDE
  \bex
  \left\{\begin{array}{ll}
    d \Psi(s)=A\Psi(s) d s+C\Psi(s) d W(s), & s \geq 0 \\
    \Psi(0)=I_{n}, 
  \end{array}\right.
  \eex
and system \eqref{sde} with $(u(\cdot),v(\cdot)) = (0,0)$ satisfies $\lim _{s \rightarrow \infty} \mathbb{E}\left[X(s)^{\top}X(s)\right]=0$. Define $H(\Psi(\cdot), \Lambda) = \mathbb{E}\int_{0}^{\infty}\Psi(s)^{\top}\Lambda \Psi(s) ds$, and $a(A) = log(5/4)/\|A\|_{2}$, where $\Lambda \in \mathbb{S}^{n}$ is a positive semi-definite matrix. Then,
\bex
\left\| H(\Psi(\cdot), \Lambda) \right\|_{2} \geq  \frac{1}{2}a(A) \|\Lambda\|_{2}.
\eex
\end{lemma}
\begin{proof}
The Taylor expansion of $e^{As}$ is 
\bex
e^{As} = I_{n} + \sum_{k=1}^{\infty}\frac{(As)^{k}}{k!} = I_{n} + F(s).
\eex

Hence,
\be
\label{F_abs}
\left\| F(s) \right\| \leq  \sum_{k=1}^{\infty}\frac{(\|A\|s)^{k}}{k!} = e^{\|A\|s} -1.
\ee 

Choose a $x\in \mathbb{R}^{n}$ which satisfies $x^{\top} \Lambda x = \|\Lambda\| |x|^{2}$. Then, by Jensen's inequality and \eqref{F_abs} we get
\be 
\label{H_ineq}
\begin{aligned}
x^{\top} H x &\geq \int_{0}^{\infty} x^{\top} (\mathbb{E}\Psi(s))^{\top} \Lambda (\mathbb{E}\Psi(s)) x ds \\ & =
\int_{0}^{\infty} x^{\top} e^{A^{\top}s} \Lambda e^{As} x ds\\ & \geq \int_{0}^{a(A)} x^{\top} e^{A^{\top}s} \Lambda e^{As} x ds \geq \frac{1}{2}a(A) \|\Lambda\| |x|^{2},
\end{aligned}
\ee 
which the last inequality can be derived from \citep[Lemma 7]{cui2023lyapunov}.
\end{proof}

\begin{lemma}
\label{L_v_in_D}
For $\delta_{3} >0$ and $\hat{L}_v^{(k)} \in \mathcal{D}_{\delta_{3}}\left(P_v^{(\ast)}\right)$, if $\left\| \Delta L_v\right\|_{\infty} < f(\delta_{3})$, where $\Delta L_v^{(k)} = \hat{L}_v^{(k)}- L_v^{(k)}$, and $f(\delta_{3})$ is defined in \eqref{f_3def}, then $\hat{L}_v^{(k+1)} \in \mathcal{D}_{\delta_{3}}\left(P_v^{(\ast)}\right)$.
\end{lemma}
\begin{proof}
  From Lemma \ref{L_v_in_D}, $\hat{L}_v^{(k)} \in \mathcal{D}_{\delta_{3}}\left(P_v^{(\ast)}\right)$ for any $k = 0,1,2,\cdots$.
  By Algorithm \ref{algo3}, $\hat{P}_{v}^{(k+1)}$ is the solution of
  \be
  \label{out_lyak+1}
  \begin{aligned}
  A_{(\ast)}^{(k+1)^{\top}}\hat{P}_{v}^{(k+1)} + \hat{P}_{v}^{(k+1)}A_{(\ast)}^{(k+1)} + C_{(\ast)}^{(k+1)^{\top}}\hat{P}_{v}^{(k+1)}C_{(\ast)}^{(k+1)} + L_u^{(k+1)^{\top}}RL_u^{(k+1)} + Q - \gamma^{2}\hat{L}_v^{(k)^{\top}}\hat{L}_v^{(k)} = 0,
  \end{aligned}
  \ee
  where $A_{(\ast)}^{(k+1)} = A + B_{1} L_u^{(k+1)} + B_2 \hat{L}_v^{(k)}$, $C_{(\ast)}^{(k+1)} = C + DL_u^{(k+1)}$
   and $L_u^{(k+1)} = - \left[M\left(\hat{P}_{v}^{(k+1)}\right)\right]_{u u}^{-1}\left[M\left(\hat{P}_{v}^{(k+1)}\right)\right]_{u x}$.
  
  For the former iteration, it can be written as
  \be
  \label{out_lyak}
  \begin{aligned}
    A_{(\ast)}^{(k)^{\top}}\hat{P}_{v}^{(k)} + \hat{P}_{v}^{(k)}A_{(\ast)}^{(k)} + C_{(\ast)}^{(k)^{\top}}\hat{P}_{v}^{(k)}C_{(\ast)}^{(k)} + L_u^{(k)^{\top}}RL_u^{(k)} + Q - \gamma^{2}\hat{L}_v^{(k-1)^{\top}}\hat{L}_v^{(k-1)} = 0,
  \end{aligned}
  \ee

  Subtracting \eqref{out_lyak} from \eqref{out_lyak+1} yields
  \bex
  \begin{aligned}
  & A_{(\ast)}^{(k+1)^{\top}}\left(\hat{P}_{v}^{(k+1)} - \hat{P}_{v}^{(k)}\right) + \left(\hat{P}_{v}^{(k+1)} - \hat{P}_{v}^{(k)}\right)A_{(\ast)}^{(k+1)} + C_{(\ast)}^{(k+1)^{\top}}\left(\hat{P}_{v}^{(k+1)} - \hat{P}_{v}^{(k)}\right)C_{(\ast)}^{(k+1)}  + \left(L_u^{(k+1)} - L_u^{(k)}\right)^{\top} \left(R+D^{\top}\hat{P}_{v}^{(k)}D\right)\\ & \times \left(L_u^{(k+1)} - L_u^{(k)}\right) + \gamma^{2}\left(L_v^{(k)} -\hat{L}_v^{(k-1)} \right)^{\top}\left(L_v^{(k)} -\hat{L}_v^{(k-1)} \right) - \gamma^{2}\left(\hat{L}_v^{(k)}- L_v^{(k)}\right)^{\top}\left(\hat{L}_v^{(k)}- L_v^{(k)}\right) = 0,
  \end{aligned}
  \eex
  where $L_v^{(k)} = - \left[M\left(\hat{P}_{v}^{(k)}\right)\right]_{v v}^{-1}\left[M\left(\hat{P}_{v}^{(k)}\right)\right]_{v x}$.
  
  Let $	\Psi^{k+1} = \left\{\Psi^{k+1}(s); s \geq 0\right\}$ be the solution to the matrix SDE
  \bex
  \left\{\begin{array}{ll}
    d \Psi^{k+1}(s)=A_{(\ast)}^{(k+1)}\Psi^{k+1}(s) d s+C_{(\ast)}^{(k+1)}\Psi^{k+1}(s) d W(s), & s \geq 0 \\
    \Psi^{k+1}(0)=I_{n}.
  \end{array}\right.
  \eex
  
  Since $I_{n} \otimes A_{(\ast)}^{(k+1)^{\top}} +A_{(\ast)}^{(k+1)^{\top}} \otimes I_{n} + C_{(\ast)}^{(k+1)^{\top}} \otimes C_{(\ast)}^{(k+1)^{\top}}$ is Hurwitz, it follows from \citep[Theorem 3.2.3]{sun2020stochastic} that $\hat{P}_{v}^{(k+1)} - \hat{P}_{v}^{(k)}$ satisfies the inequality
  \be
  \label{ineq_pvk}
  \begin{aligned}
  \hat{P}_{v}^{(k+1)} - \hat{P}_{v}^{(k)} \geq & \mathbb{E} \int_{0}^{\infty} \Psi^{k+1}(s)^{\top} \left(\gamma^{2}\left(L_v^{(k)} -\hat{L}_v^{(k-1)} \right)^{\top}\left(L_v^{(k)} -\hat{L}_v^{(k-1)} \right)\right)\Psi^{k+1}(s) ds \\ &- \mathbb{E} \int_{0}^{\infty} \Psi^{k+1}(s)^{\top} \left(\gamma^{2}\left(\hat{L}_v^{(k)}- L_v^{(k)}\right)^{\top}\left(\hat{L}_v^{(k)}- L_v^{(k)}\right)\right)\Psi^{k+1}(s) ds
  \end{aligned}
  \ee

  By Lemma \ref{lem_Ha} we have
  \be
  \label{ineq_psik}
  \begin{aligned}
  & \left\| \mathbb{E} \int_{0}^{\infty} \Psi^{k+1}(s)^{\top} \left(\gamma^{2}\left(L_v^{(k)} -\hat{L}_v^{(k-1)} \right)^{\top}\left(L_v^{(k)} -\hat{L}_v^{(k-1)} \right)\right)\Psi^{k+1}(s) ds\right\| \\ \geq &\frac{log(5/4)}{2\left\| A_{(\ast)}^{(k+1)}\right\|} \left\| \gamma^{2}\left(L_v^{(k)} -\hat{L}_v^{(k-1)} \right)^{\top}\left(L_v^{(k)} -\hat{L}_v^{(k-1)} \right)\right\| .
  \end{aligned}
  \ee

  Taking the trace on both sides of the inequality \eqref{ineq_pvk}, by \eqref{ineq_psik}, Lemma \ref{out_lem1} and \citep[Theorem 1]{mori1988comments}  we obtain
  \bex
  \begin{aligned}
  Tr\left(\hat{P}_{v}^{(k+1)} - \hat{P}_{v}^{(k)}\right) \geq &\frac{log(5/4)}{2n\sqrt{n}b(\delta_3)\left\| A_{(\ast)}^{(k+1)}\right\|}Tr\left(P_v^{(\ast)}  - \hat{P}_v^{(k)} \right) \\ &- Tr\left(\mathbb{E} \int_{0}^{\infty} \Psi^{k+1}(s)^{\top} \Psi^{k+1}(s) ds\right) \left\|\gamma^{2}\left(\hat{L}_v^{(k)}- L_v^{(k)}\right)^{\top}\left(\hat{L}_v^{(k)}- L_v^{(k)}\right)\right\|_{2}.
  \end{aligned}
  \eex

  Then, we have
\be
\label{Trast}
\begin{aligned}
Tr\left(P_v^{(\ast)} - \hat{P}_v^{(k+1)}\right)  =& Tr\left(P_v^{(\ast)} - \hat{P}_v^{(k)}\right) - Tr\left(\hat{P}_v^{(k+1)} - \hat{P}_v^{(k)}\right) \\  \leq & \left(1-  \frac{log(5/4)}{2n\sqrt{n} b(\delta_3)\left\| A_{(\ast)}^{(k+1)}\right\|}\right)Tr\left(P_v^{(\ast)} - \hat{P}_v^{(k)}\right)  \\ &+ Tr\left(\mathbb{E} \int_{0}^{\infty} \Psi^{k+1}(s)^{\top} \Psi^{k+1}(s) ds\right) \left\|\gamma^{2}\left(\hat{L}_v^{(k)}- L_v^{(k)}\right)^{\top}\left(\hat{L}_v^{(k)}- L_v^{(k)}\right)\right\|_{2}
\end{aligned}
\ee

Let 
\bex 
f_{1}(\hat{L}_v^{(k)}) = \frac{log(5/4)}{2n\sqrt{n} b(\delta_3)\left\| A_{(\ast)}^{(k+1)}\right\|},\quad f_{2}\left(\hat{L}_v^{(k)}\right) = Tr\left(\mathbb{E} \int_{0}^{\infty} \Psi^{k+1}(s)^{\top} \Psi^{k+1}(s) ds\right).
\eex

Since $f_{1}\left(\hat{L}_v^{(k)}\right)$ and $f_{2}\left(\hat{L}_v^{(k)}\right)$ are continuous with respect to $\hat{L}_v^{(k)}$,
\bex 
\underline{f}_{1}(\delta_{3}) = \inf\limits_{\hat{L}_v^{(k)} \in \mathcal{D}_{\delta_{3}}\left(P_v^{(\ast)}\right) } f_{1}\left(\hat{L}_v^{(k)}\right) >0, \quad \bar{f}_{2}(\delta_{3})=\sup\limits_{\hat{L}_v^{(k)} \in \mathcal{D}_{\delta_{3}}\left(P_v^{(\ast)}\right) }f_{2}\left(\hat{L}_v^{(k)}\right) < \infty.
\eex 

It follows from \eqref{Trast} that if 
$
\left\|\left(\hat{L}_v^{(k)}- L_v^{(k)}\right)\right\|_{2} < \sqrt{\frac{\underline{f}_{1}(\delta_{3})\delta_{3}}{\gamma^{2}\bar{f}_{2}(\delta_{3})}},
$
we have 
$
Tr\left(P_v^{(\ast)} - \hat{P}_v^{(k+1)}\right) < \delta_{3}.
$

There exists $f_{3}(\delta_{3})>0$ such that when $\left\|\left(\hat{L}_v^{(k)}- L_v^{(k)}\right)\right\|_{2} \leq f_{3}(\delta_{3})$, $(L_{u},\hat{L}_{v}^{(k+1)})$ is a stabilizer of system \eqref{sde}. Therefore, if 
\be 
\label{f_3def}
\left\|\left(\hat{L}_v- L_v\right)\right\|_{\infty} < \min \left\{f_{3}(\delta_{3}),\sqrt{\frac{\underline{f}_{1}(\delta_{3})\delta_{3}}{\gamma^{2}\bar{f}_{2}(\delta_{3})}}\right\} = : f(\delta_{3}),
\ee
we have $\hat{L}_v^{(k+1)} \in \mathcal{D}_{\delta_{3}}\left(P_v^{(\ast)}\right)$.

\end{proof}

\begin{theorem}
\label{main_thm_out}
For $\delta_{3} >0$ and $\hat{L}_v^{(0)} \in \mathcal{D}_{\delta_{3}}\left(P_v^{(\ast)}\right)$, if $\left\| \Delta L_v\right\|_{\infty} < f(\delta_{3})$ , then the following small-disturbance ISS holds: 
\be
\left\|P_v^{(\ast)} - \hat{P}_v^{(k)} \right\|_{F} \leq \beta_{1}\left(\left\|P_v^{(\ast)} - \hat{P}_v^{(0)}\right\|_{F}, k\right) + \kappa_{1} \left(\left\| \Delta L_v\right\|_{\infty} \right),
\ee
where $\beta_{1}(\cdot,\cdot)$ is a class $\mathcal{KL}$ function, and $\kappa_{1}(\cdot)$ is a class $\mathcal{K} function$ .
\end{theorem}

\begin{proof}

It follows from Lemma \ref{L_v_in_D} that $\hat{L}_v^{(k)} \in \mathcal{D}_{\delta_{3}}\left(P_v^{(\ast)}\right)$ for any $k = 0,1,2,\cdots$. From \eqref{Trast}, we have
\be 
\label{Trast_2}
\begin{aligned}
Tr\left(P_v^{(\ast)} - \hat{P}_v^{(k+1)}\right)  \leq & \left(1-  \underline{f}_{1}(\delta_{3})\right)Tr\left(P_v^{(\ast)} - \hat{P}_v^{(k)}\right) + \bar{f}_{2}(\delta_{3})\left\|\gamma^{2}\left(\hat{L}_v^{(k)}- L_v^{(k)}\right)^{\top}\left(\hat{L}_v^{(k)}- L_v^{(k)}\right)\right\|_{2}.
\end{aligned}
\ee

Finally, similar to (ii) in Lemma \ref{main_lemma1}, the proof of the theorem can be completed by repeating the above inequality.
\end{proof}

\begin{remark}
  In practical implementation, the computations in Algorithm \ref{algo2} need to be approximated by collecting state data at discrete time points and using methods such as the EM algorithm. Additionally, the conditional expectation must be approximated using Monte Carlo methods, which introduces error in the iterative process. According to \cite{hutzenthaler2013divergence} and \cite{cao2024analyzing}, it is known that the approximation error can be reduced by decreasing the discretization interval and increasing the number of samples. Therefore, $\left\| \Delta M\right\|_{\infty} < \delta_{2}$ in Theorem \ref{thm_main} and $\left\| \Delta L_v\right\|_{\infty} < f(\delta_{3})$ in Theorem \ref{main_thm_out} can be satisfied.
\end{remark}

\section{Numerical Experiments}

In this section, we will verify the effectiveness of Algorithm \ref{algo2}.

\textbf{Example 1.} (2-dimensional case)  Let the dimension $n=2$, $m_1 = m_2 =1$ and set
\bex
\begin{aligned}
  & A = \begin{bmatrix} -1.2115&0.7141\\0.8597&-1.0757 \end{bmatrix},\quad B_{1}= \begin{bmatrix} 0.6052\\0.6433 \end{bmatrix},\quad B_{2}=  \begin{bmatrix} 0.3281\\0.8746 \end{bmatrix},\\
  & C = \begin{bmatrix} 0.0743&0.0545\\0.0935&0.0397 \end{bmatrix},\quad D=\begin{bmatrix} 0.0774\\0.0118 \end{bmatrix}.
\end{aligned}
\eex

The coefficients in the quadratic performance index are
\bex
\begin{aligned}
Q = \begin{bmatrix} 0.2&0\\0&0.2 \end{bmatrix},\quad R=0.24 ,\quad \gamma=0.5.
\end{aligned}
\eex

Now, we will solve this problem using the presented  model-free algorithm. The initial state is selected as $x = [2,3]^{\top}$. Since the dimension of the parameter to be estimated in \eqref{ls_4} is $N = \frac{n(n+1)}{2} + m_{1}n + m_{2}n + \frac{m(m+1)}{2} = 8$, we divided the time interval $[0, 2]$ into $N = 8$ equal parts. Then, to ensure that the full column rank condition in Lemma \eqref{lem_rank} is satisfied, we make the exploration noises $e_{u} = 0.1sin(10t) + \sqrt{\lambda_{1}R^{-1}} \xi$, $e_{v} = 0.1cos(10t) + \sqrt{\lambda_{2}\gamma^{-2}} \xi$ in Algorithm \ref{algo2}, where $\lambda_{1}=\lambda_{2} =0.05$, and $\xi$ is a standard normal random vector, to ensure sufficient exploration \citep{bian2016adaptive, wang2020reinforcement, chen2023online}. The values of stop criteria $\epsilon_{1}$ and $\epsilon$ in Algorithm \ref{algo2} are set to $10^{-5}$ and $10^{-5}$.

By implementing Algorithm \ref{algo2}, Figure \ref{fig1} shows the parameter iterations of the value function in the learning process. It is evident that the inner loop (steps 5-9) of the Algorithm \ref{algo2} is convergent after a finite number of steps. The convergence results for $\hat{P}_v^{(k+1,j)}, k=0, 1, \dots$ are given by matrices
\bex
\begin{aligned}
  \hat{P}_u^{(1, \ast)}= \begin{bmatrix} 0.1034&0.0585\\0.0585&0.1058 \end{bmatrix},\quad \hat{P}_u^{(2, \ast)}=\begin{bmatrix} 0.1363&0.0977\\0.0977&0.1465 \end{bmatrix},
\end{aligned}
\eex
\bex
\begin{aligned}
  \hat{P}_u^{(3, \ast)}=\begin{bmatrix} 0.1454&0.1077\\0.1077&0.1559 \end{bmatrix}, \quad \hat{P}_v^{(4, \ast)}=\begin{bmatrix} 0.1460&0.1084\\0.1084&0.1565 \end{bmatrix}.
\end{aligned}
\eex

Then, since the convergence error $\Vert \hat{L}_{v}^{(6)} -\hat{L}_{v}^{(5)}\Vert \leq \epsilon$, we obtain
\bex
\begin{aligned}
  \hat{L}_{u}^{(\ast)} = \hat{L}_{u}^{(6,\ast)}= \begin{bmatrix} -0.7082&-0.6618 \end{bmatrix}, \quad \hat{L}_{v}^{(\ast)}= \hat{L}_{v}^{(6)}=\begin{bmatrix} 0.6021&0.6673 \end{bmatrix}, 
\end{aligned}
\eex
and
\bex
\begin{aligned}
  \hat{P}^{(\ast)}=\hat{P}_u^{(6, \ast)} = \begin{bmatrix} 0.1460&0.1084\\0.1084&0.1565 \end{bmatrix}.
\end{aligned}
\eex

For comparison, we take the result 
\bex
\begin{aligned}
  P=\begin{bmatrix} 0.1473&	0.1041\\
    0.1041&	0.1661 \end{bmatrix}
\end{aligned}
\eex
obtained from the model-based PI Algorithm \ref{algo1} as the true value. The error between the converged value $ \hat{P}^{(\ast)}$ from Algorithm \ref{algo2} and the true value $P$ is $\Vert \hat{P}^{(\ast)} - P\Vert = 0.0114$, showing a good convergence effect. 

Finally, to demonstrate that the feedback control pair $(u,v)= \left(\hat{L}_{u}^{(\ast)}X,\hat{L}_{v}^{(\ast)}X \right)$ is stabilizing, we present the state trajectory of the system under this input control in Figure \ref{ex1-xstable}.

\begin{figure}[!htp]
  \centering 
  \includegraphics[width=0.8\textwidth]{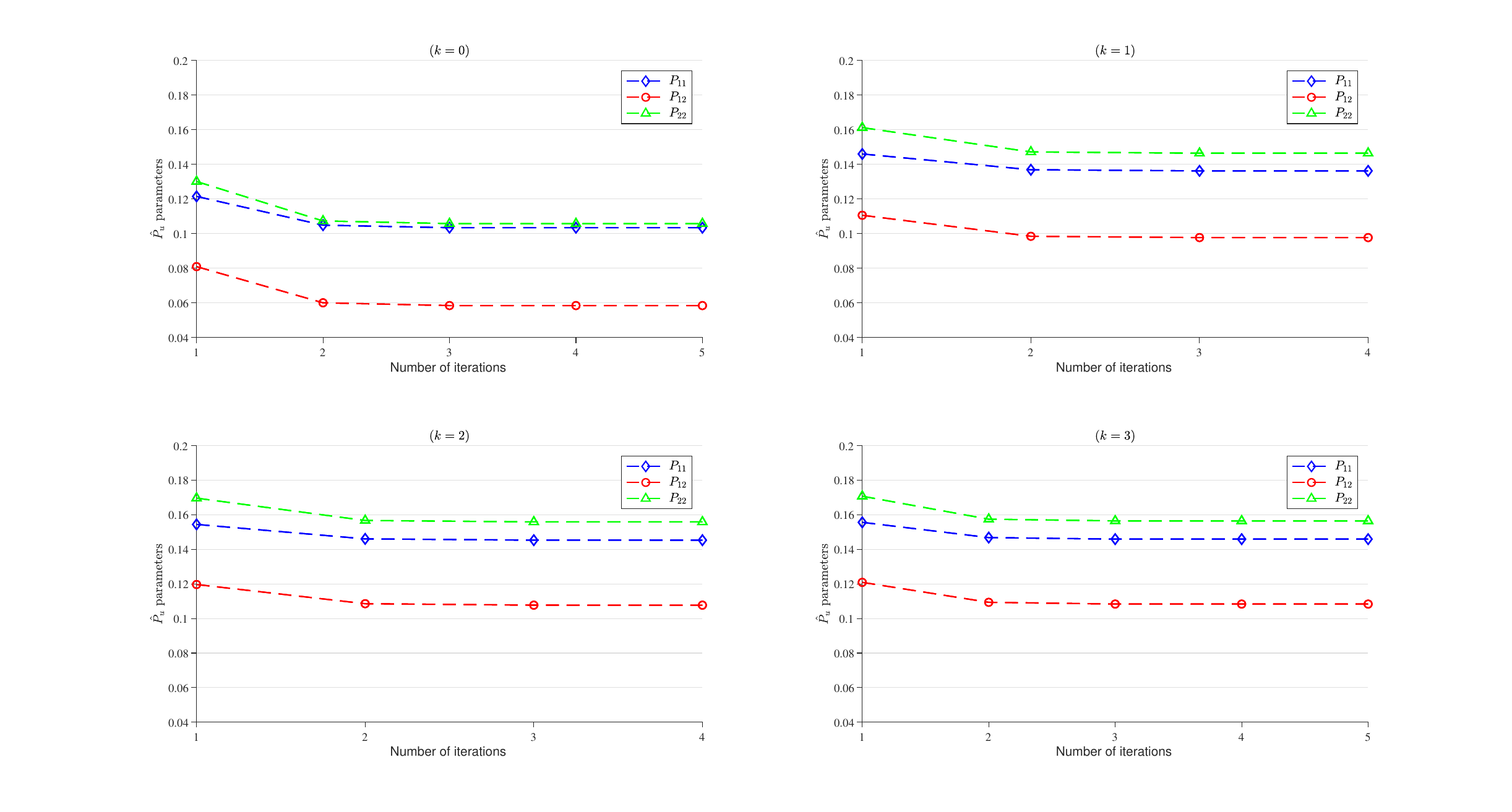}
  \caption{Convergence of the $\hat{P}_u$ parameters of the inner loop $\left(k=0,1,2,3,4\right)$.}
  \label{fig1} 
\end{figure}
\begin{figure}[!htp]
  \centering 
  \includegraphics[width=0.8\textwidth]{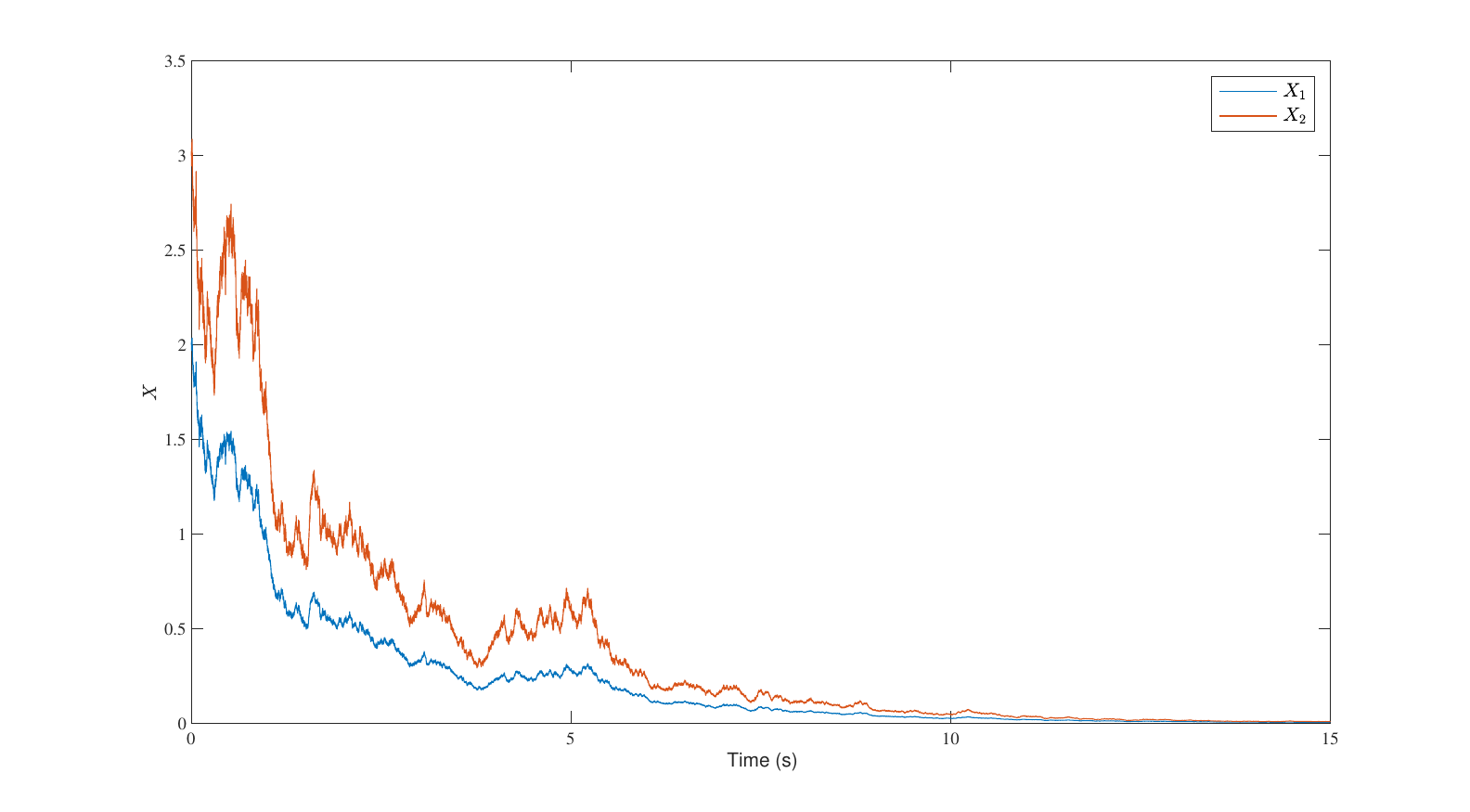}
  \caption{The state trajectory of system \eqref{sde} under the input control pair $(u,v)= \left(\hat{L}_{u}^{(\ast)}X,\hat{L}_{v}^{(\ast)}X \right)$.}
  \label{ex1-xstable} 
\end{figure}

\textbf{Example 2.} (4-dimensional case) We consider a linear two-mass spring system that has uncertain stiffness \citep{feng2010iterative} and suppose that a multiplicative noise perturbs it. The model of the stochastic spring system can be described by system \eqref{sde} with matrices
\bex
\begin{aligned}
  & A = \begin{bmatrix} 0&0&1&0\\0&0&0&1\\-1.25&1.25&0&0\\1.25&-1.25&0&0\end{bmatrix},\quad B_{1}= \begin{bmatrix} 0\\0\\0\\1 \end{bmatrix},\quad B_{2}=  \begin{bmatrix} 0\\0\\1\\0 \end{bmatrix},\\
  & C = \begin{bmatrix} 0&0&0&0\\0&0&0&0\\-0.25&0.25&0&0\\0.25&-0.25&0&0 \end{bmatrix},\quad D=\begin{bmatrix} 0\\0\\0\\0.2 \end{bmatrix}.
\end{aligned}
\eex
Then the presented RL method is applied to solve the stochastic LQZSG problem with weighting matrices
\bex
\begin{aligned}
Q = \begin{bmatrix} 1&0&0&0\\0&1&0&0\\0&0&0&0\\0&0&0&0 \end{bmatrix},\quad R = 1,\quad \gamma = 2.
\end{aligned}
\eex

By implementing Algorithm \ref{algo2}, it is shown that the inner loop in Algorithm \ref{algo2} satisfies the stopping criterion $\Vert \hat{P}_u^{(k+1,j)} - \hat{P}_u^{(k+1,j-1)}\Vert \leq 10^{-5}$, $k = 0,1, \dots, 8$, and convergence is obtained:
\bex
\begin{aligned}
  \hat{P}_u^{(1,\ast)} = \begin{bmatrix} 2.0795&	-0.6274&	0.5586&0.1212\\-0.6274&	2.9162&	1.5361&	1.3055\\0.5586&	1.5361&	2.5822&	0.9784\\0.1212&	1.3055&	0.9784&	1.6321\end{bmatrix}, \quad \hat{P}_u^{(2,\ast)}= \begin{bmatrix} 2.5583&	-1.4013&	0.0200&	-0.3751\\-1.4013&	5.4209&	4.0395&	2.7529 \\ 0.0200&	4.0395&	5.5042&	2.4334\\-0.3751&	2.7529&	2.4334&	2.6540 \end{bmatrix},
\end{aligned}
\eex
\bex
\begin{aligned}
  \hat{P}_u^{(3,\ast)}=\begin{bmatrix} 2.9851&-2.2777&-0.7285&-0.9177\\-2.2777&7.7588&6.2857&4.1544\\ -0.7285&6.2857&7.8694&3.7968\\-0.9177&4.1544&3.7968&3.6121 \end{bmatrix}, \quad \hat{P}_u^{(4,\ast)}=\begin{bmatrix} 3.2803&-2.8847&-1.2576&-1.2985\\-2.8847&9.2831&7.7383&5.0961\\ -1.2576&7.7383&9.3697&4.7093\\-1.2985&5.0961&4.7093&4.2636 \end{bmatrix},
\end{aligned}
\eex

\bex
\begin{aligned}
  \hat{P}_u^{(5,\ast)}=\begin{bmatrix} 3.4157&	-3.1633&	-1.5068&	-1.4760\\
    -3.1633&	9.9803&	8.4148&	5.5367\\
    -1.5068&	8.4148&	10.0770&	5.1456\\
    -1.4760&	5.5367&	5.1456&	4.5725 \end{bmatrix}, \cdots, \hat{P}_u^{(9,\ast)}=\begin{bmatrix}3.4522&	-3.2384	&-1.5757&	-1.5246\\
      -3.2384&	10.1695&	8.6024&	5.6584\\
      -1.5757&	8.6024&	10.2767&	5.2692\\
      -1.5246&	5.6584&	5.2692&	4.6589\end{bmatrix},
\end{aligned}
\eex

From Figure \ref{fig2}, it is shown that the outer loop of Algorithm \ref{algo2} is terminated after $k = 9$ iterations.

For comparison, we take the result 
\bex
\begin{aligned}
  P=\begin{bmatrix}3.4025&	-3.1333&	-1.4814&	-1.4601\\
    -3.1333&	9.9070&	8.3560&	5.5014\\
    -1.4814&	8.3560&	10.0285&	5.1207\\
    -1.4601&	5.5014&	5.1207&	4.5561 \end{bmatrix}
\end{aligned}
\eex
obtained from the model-based PI Algorithm \ref{algo1} as the true value. The error between the converged value $ \hat{P}_u^{(9,\ast)}$ from Algorithm \ref{algo2} and the true value $P$ is $\Vert \hat{P}_u^{(9,\ast)} - P\Vert = 0.6374$. It can be seen that, in the presence of estimation errors, the algorithm can converge to a small neighborhood around the true value. 

Similarly, to demonstrate that the feedback control pair $(u,v)= \left(\hat{L}_{u}^{(\ast)}X,\hat{L}_{v}^{(\ast)}X \right)$ is stabilizing, we present the state trajectory of the system under this input control in Figure \ref{ex2-xstable}.

\begin{figure}[!htp]
  \centering 
  \includegraphics[width=0.8\textwidth]{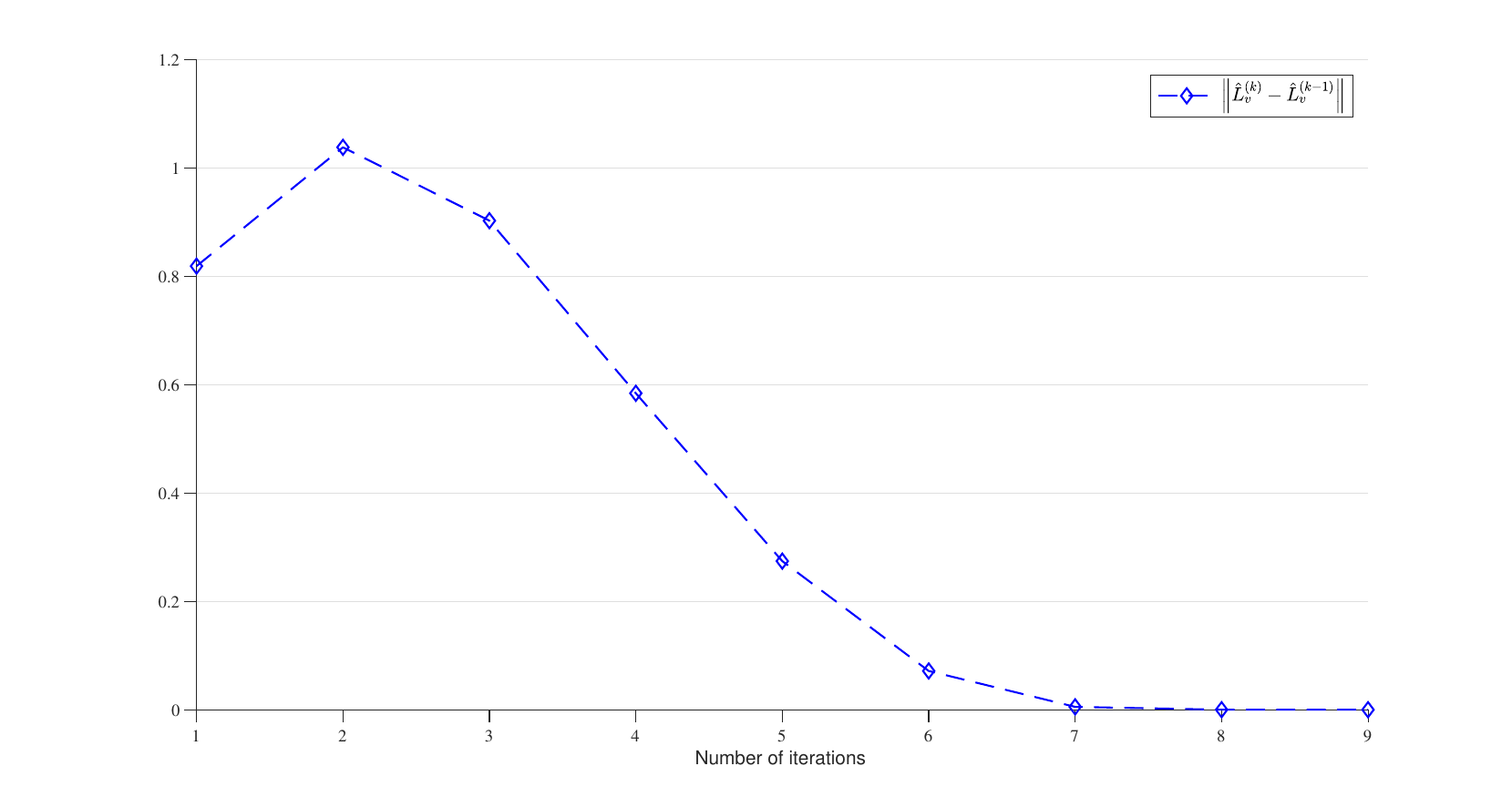}
  \caption{Convergence of the $\hat{L}_v^{(k)}$ parameters of the outer loop.}
  \label{fig2} 
\end{figure}

\begin{figure}[!htp]
  \centering 
  \includegraphics[width=0.8\textwidth]{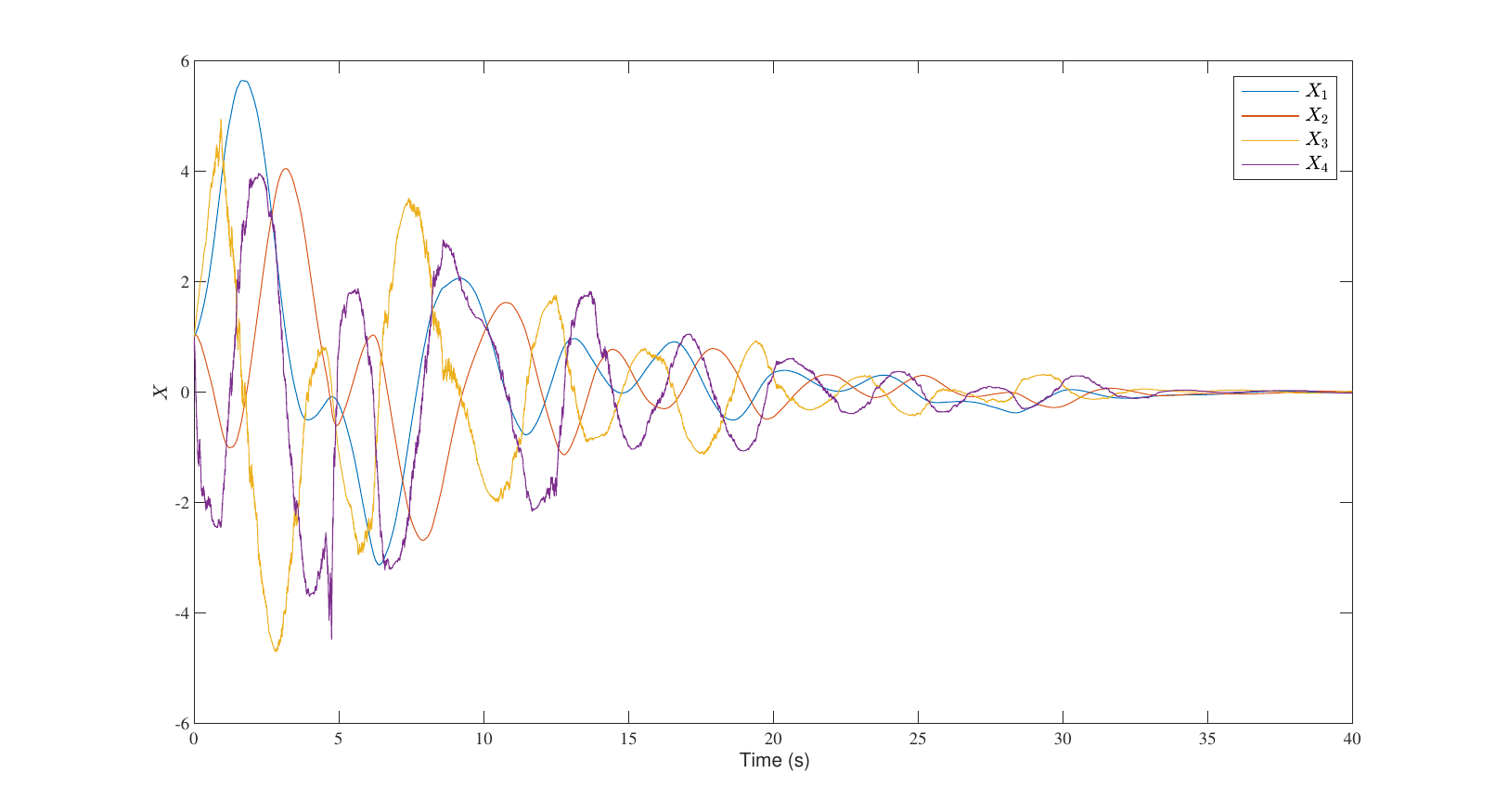}
  \caption{The state trajectory of system \eqref{sde} under the input control pair $(u,v)= \left(\hat{L}_{u}^{(\ast)}X,\hat{L}_{v}^{(\ast)}X \right)$.}
  \label{ex2-xstable} 
\end{figure}









\bibliographystyle{elsarticle-num-names} 
\bibliography{References}
\end{document}